\documentclass{amsart}
\usepackage[latin1]{inputenc}
\usepackage[T1]{fontenc}
\usepackage{amssymb, amsmath, color, textcomp, url,tikz, mathtools}
\usetikzlibrary{matrix,arrows}
\usetikzlibrary{fit}
\usetikzlibrary{positioning}
\usepackage[labelformat=empty]{caption}

\usepackage[normalem]{ulem}
\usepackage{soul}
\usepackage{mdwlist}

\usepackage{tikz-cd}

\RequirePackage{ifpdf}
\ifpdf
   \usepackage[pdftex]{hyperref}
\else
   \usepackage[hypertex]{hyperref}
\fi

\usepackage{enumerate,xspace}
\usepackage{chngcntr, csquotes} 
\theoremstyle{plain}
\newtheorem{theorem}{Theorem}[section]
\newtheorem{cor}[theorem]{Corollary}
\newtheorem{prop}[theorem]{Proposition}
\newtheorem{lemma}[theorem]{Lemma}

\newtheorem*{theosin}{Theorem}
\newcounter{proofcount}
\AtBeginEnvironment{proof}{\stepcounter{proofcount}}

\newtheorem*{claim*}{Claim}
\makeatletter                  
\@addtoreset{claim}{proofcount}
\makeatother                   

\usepackage{pdfpages}

\newenvironment{claimproof*}[1][Proof of Claim.] 
{%
	\proof[#1]%
	
}
{%
	\endproof%
}

\theoremstyle{definition}
\newtheorem{remark}[theorem]{Remark}
\newtheorem{fact}[theorem]{Fact}
\newtheorem{definition}[theorem]{Definition}
\newtheorem{example}[theorem]{Example}

\newtheorem*{quest}{Question}

\newcommand{\nc}{\newcommand}

\nc{\Z}{\mathbb{Z}}
\nc{\Q}{\mathbb{Q}}
\nc{\N}{\mathbb{N}}
\nc{\F}{\mathbb{F}}
\nc{\UU}{\mathbb{U}}
\nc{\C}{\mathbb{C}}

\nc{\M}{\mathcal{M}}
\nc{\R}{\mathcal{R}}
\nc\LL{\mathcal L}
\nc\II{\mathcal I}

\nc{\stt}{\operatorname{St}}
\nc{\stab}{\operatorname{Stab}}
\nc{\GO}[1]{G_{#1}^{00}}
\nc{\sbgp}[1]{\langle\xspace {#1}\xspace\rangle}
\nc{\Conn}[1]{\langle\xspace {X}\xspace\rangle^{00}_{#1}}
\nc{\band}[1]{\bar d_{\mathcal{#1}}}
\nc\Def{\operatorname{Def}}

\nc{\dcl}{\operatorname{dcl}}

\nc{\acl}{\operatorname{acl}}

\nc{\nf}[1]{_{\mid {#1}}}
\nc{\restr}[1]{\xspace_{\upharpoonright {#1}}}

\nc\inv{ ^{-1}}

\nc{\tp}{\operatorname{tp}}
\nc\Spec{S^\mathrm{t}}
\nc\HS{S^{\mathrm{h}_M}}
\nc\U{\operatorname{U}}
\nc{\cf}{\text{cf.\,}}
\nc{\eg}{\text{e.g. }}

\nc{\InvR}[1]{S_{#1}(\bar{\R})^{\rm inv}_{R}}
\nc{\CohR}[1]{S_{#1}(\bar{\R})^{\rm fs}_{R}}
\nc{\InvRt}[1]{S^{\rm t}_{#1}(\bar{\R})^{\rm inv}_{R}}

\nc{\Inv}[1]{S_{#1}(\bar{\M})^{\rm inv}_{M}}
\nc{\Invh}[1]{S^{{\rm h}_M}_{#1}(\bar{\M})^{\rm inv}_{M}}
\nc{\Coh}[1]{S_{#1}(\bar{\M})^{\rm fs}_{M}}
\nc\bM{\overline{M}}

\def\Ind#1#2{#1\setbox0=\hbox{$#1x$}\kern\wd0\hbox to
  0pt{\hss$#1\mid$\hss} \lower.9\ht0\hbox to
  0pt{\hss$#1\smile$\hss}\kern\wd0}
\def\Notind#1#2{#1\setbox0=\hbox{$#1x$}\kern\wd0\hbox to
  0pt{\mathchardef\nn="0236\hss$#1\nn$\kern1.4\wd0\hss}\hbox to
  0pt{\hss$#1\mid$\hss}\lower.9\ht0 \hbox to
  0pt{\hss$#1\smile$\hss}\kern\wd0}

\def\indip{\mathop{\ \ \hbox to 0pt{\hss$\mid^{\hbox to
0pt{$\scriptstyle P$\hss}}$\hss}
\lower4pt\hbox to 0pt{\hss$\smile$\hss}\ \ }}
\def\nindip{\mathop{\ \ \hbox to 0pt{\hss$\!\not{\mid}^{\hbox to
0pt{$\scriptstyle\, P$\hss}}$\hss}
\lower4pt\hbox to 0pt{\hss$\smile$\hss}\ \ }}

\title[]{Spectral spaces in o-minimal and other NIP theories}

\address{Departamento de \'Algebra, Geometr\'ia y Topolog\'ia; Facultad de Matem\'aticas;
	Universidad Complutense de Madrid; 28040 Madrid, Spain}

\email{eliasbaro@pdi.ucm.es}
\email{josefer@ucm.es}
\email{dpalacin@ucm.es}

\date{\today}

\author{El\'ias Baro, Jos\'e F. Fernando, and Daniel Palac\'in}

\thanks{All authors are supported by Spanish STRANO PID2021-122752NB-I00 and Grupos UCM 910444. First and second authors were also supported by STRANO MTM2017-82105-P. The third author was also supported by the contract 2020-T1/TIC-20313 from Community of Madrid, as well as by the Deutsche Forschungsgemeinschaft
	(DFG, German Research Foundation) - Project number 2100310301, part of the
	ANR-DFG program GeoMod.
}


\begin{document}
\maketitle
\begin{abstract} We study some model-theoretic notions in NIP by means of spectral topology. In the o-minimal setting we relate the o-minimal spectrum with other topological spaces such as the real spectrum and the space of infinitesimal types of Peterzil and Starchenko. In particular, we prove for definably compact groups  that the space of closed points is homeomorphic to the space of infinitesimal types. We also prove that with the spectral topology the set of invariant types concentrated in a definably compact set is a normal spectral space whose closed points are the finitely satisfiable types. 	

On the other hand, for arbitrary NIP structures we equip the set of invariant types with a new topology, called the {\em honest topology}. With this topology the set of invariant types is a normal spectral space whose closed points are the finitely satisfiable ones, and the natural retraction from invariant types onto finitely satisfiable types coincides with Simon's $F_M$ retraction.
\end{abstract}

\section{Introduction}

Spectral spaces  (see Definition \ref{def:spectralspace}) constitute a topological framework related with several areas of mathematics, specially with algebraic geometry via the Zariski spectrum of a ring (and afterwards, following the abstract development of Grothen\-dieck). In the late 1970s, the real spectra of a ring was introduced by Coste and Roy in their foundational paper \cite{CR82} as a substitute in real algebraic geometry for the Zariski spectra of a ring in ordinary algebraic geometry. Recall that if $R$ is a real closed field, $V \subset R^n$ is a real algebraic set and $\mathcal{R}(V)$ denotes the ring of regular functions on $V$, then the real spectra $\text{Sper}(\mathcal{R}(V))$ is a spectral space that reflects the semialgebraic properties of $V$. For example, there is a natural injection of $V$ into $\text{Sper}(\mathcal{R}(V))$, which is continuous with respect to the Euclidean topology of $V$. Moreover, we can identify canonically every semialgebraic subset $S$ of $V$ with a constructible subset $\widetilde{S}$ of  $\text{Sper}(\mathcal{R}(V))$ such that $\widetilde{S}\cap V=S$ (see \cite[Prop.\,7.2.2]{BCR98}). The space $\text{Sper}(\mathcal{R}(V))$ is an object whose (standard) topological properties yield information about (semialgebraic) topological properties of $V$. For example, a semialgebraic subset $S$ of $V$ is semialgebraically connected if and only if $\widetilde{S}$ is connected. Or more notably, it is possible to construct a (standard) sheaf theory of rings on $\text{Sper}(\mathcal{R}(V))$  whose sections are the semialgebraic continuous functions. In general, the space $\text{Sper}(\mathcal{R}(V))$ is quasi-compact but it is not Hausdorff. Nonetheless, since it is normal, the subset of its closed points constitutes a quasi-compact Hausdorff subspace (see Fact \ref{fact:max}).

Spectral spaces within model theory were first considered by Pillay \cite{aP88b}, who introduced the so-called o-minimal spectra of definable sets. Before proceeding, we first recall some basic model-theoretic concepts and fix some notation. 
Given a first order structure  $\mathcal{M}$ in a language $\LL$ and an $\LL$-formula $\phi(x)$, possibly with parameters, we denote by $\phi(M)=\{a\in M^{n}:  \phi(a) \text{ holds}\}$ the definable subset of $M^{n}$ given by $\phi(x)$, where $x$ is an $n$-tuple of variables. We will say that $\phi(x)$ is an $\LL_M$-formula, denoted by $\phi(x)\in\LL_M$, when the parameters are from $M$. We shall not distinguish between a definable set and a formula defining it. Notice that the family  of definable subsets of $M^n$ constitutes a Boolean algebra and thus it has an associated set of ultrafilters. This set is called the set of complete $n$-types over $M$ and it is denoted by $S_n(\mathcal{M})$. One can turn $S_n(\mathcal{M})$ into a topological space by considering the Stone topology, where a base of open sets is given by the sets of the form $[Z]:=\{p\in S_n(\mathcal{M}): Z\in p \}$ for $Z\subset M^n$ definable. Given a  definable set $X\subset M^n$, we denote the set of $n$-types concentrated in $X$ by $S_X(\M):=[X]$, which with the induced Stone topology is a quasi-compact Hausdorff and totally disconnected space that contains a copy of $X$. This compactification is the natural one from a model-theoretic point of view. However, in o-minimal expansions of real closed fields Pillay \cite{aP88b} consider a thicker topology. Recall that an  $\LL$-structure $\mathcal{M}$  expanding a linear order is \emph{o-minimal} if every definable subset of $M$ is a finite union of points and intervals with end points in $M\cup \{\pm \infty\}$. 

The versatility of o-minimality has recently led to spectacular applications in several areas ranging from diophantine geometry,
non-archimedean geometry, number theory and additive combinatorics.  Real algebraic geometry has (partially) served as a guide in the development of o-minimal geometry. Influenced by the real spectra, in \cite{aP88b} the set $S_X(\mathcal{M})$  is equipped with a spectral topology whose basic open sets are the sets $[U]$ with $U\subset X$ open definable. If the o-minimal structure under consideration is an expansion $\mathcal{R}$ of a real closed field then $S_X(\mathcal{R})$ equipped with the spectral topology is  a normal spectral space that we denote by $\Spec_X(\mathcal{R})$ and called it \emph{o-minimal spectrum}. In particular, the set of closed points $\beta_X(\mathcal{R})$ is a quasi-compact and Hausdorff space and there is a natural continuous retraction
$$r: \Spec_X(\R)\rightarrow \beta_X(\R).$$
The space $\Spec_X(\R)$ has been considered by several authors mainly as a tool to  develop sheaf theories in the o-minimal setting \cite{aP88b, EGP06, BF09, EP20}. In \cite{mT99}, Tressl also considers  $\Spec_X(\R)$ and relates it with the real spectra of the ring of continuous definable functions. On the other hand, the main theme of the unpublished paper \cite{aF06} is the relationship between the specialization order and the Rudin-Keisler order. 

The purpose of this paper is two-fold. First, we relate the o-minimal spectrum with some relevant model-theoretic notions. The main results concerning this part are the following. Let $\mathcal{R}$ be an o-minimal structure expanding a real closed field, and let $G$ be a definable group. In \cite{PS17},  to study types modulo infinitesimals, Peterzil and Starchenko introduce a com\-pacti\-fication of $G$ called $S^\mu_G(\mathcal{R})$ which is related with the Samuel compactification of a topological group. Their motivation to consider this compactification is to obtain definable subgroups of $G$ as stabilizers of certains elements of $S^\mu_G(\mathcal{R})$. In the paper we show:

\begin{theosin}[Corollary \ref{C:HomeoBetaMu}] Let $G$ be a definably compact group. Then $\beta_G(\mathcal{R})$ and $S^\mu_G(\mathcal{R})$ are cano\-ni\-cally homeomorphic. 
\end{theosin}

On the other hand, the set of invariant types has played a fundamental role in the development of model theory. Invariant types have been intensively studied in o-minimality \cite{aD04,sS08}, in the dp-minimal context \cite{PS14}  and also in the more general framework of NIP \cite{HP11, pS15}. If $\mathcal{R}$ is an o-minimal expansion of a real closed field and $\bar\R$ is a sufficiently saturated elementary extension of $\R$, given an $R$-definable subset $X \subset \bar R^n$, we say that a type $p\in S_X(\bar\R)$ is $R$-invariant if for every automorphism $\sigma$ of $\bar\R$ fixing $R$ pointwise we have that $p=\sigma(p)$, where $\sigma(p):=\{\sigma(Y): Y\in p\}$. We will denote the set of $R$-invariant types as $\InvR{X}$. An important subset of $R$-invariant types is the one of finitely satisfiable ones, which we will denote by $\CohR{X}$ (see \ref{rmk:invcoh}). Recall that a type $p\in S_X(\bar\R)$ is \emph{finitely satisfiable} in $R$ if for every $Y\in p$ we have that $Y\cap R^n\neq \emptyset$.  Both $\InvR{X}$ and $\CohR{X}$ are closed subsets of $S_X(\R)$ with respect to the Stone topology. In particular, both are spectral spaces with the 
topology inherited from $S_X^{\rm t}(\bar\R)$. In Corollary \ref{C:CohBeta0} and \ref{C:CohBeta} we show that $\CohR{X}$ is closed in $S_X^{\rm t}(\bar\R)$ and the inherited topology is in fact the Stone topology. Additionally, we define $S_X(\bar \R)^{\rm bdd}$ as the set of types that contains a set $\{x\in \bar{R}^n:\|x\|<r\}$ for some positive $r\in \bar R$. In this framework we have:

\begin{theosin}[Theorem \ref{T:InvBddCoh} and Corollary \ref{C:DefCompInvBetaCoh}]Let $X\subset \bar R^n$ be a closed definable set. Then
\[
	\InvR{X} \cap \beta_X(\bar \R) \cap S_X(\bar \R)^{\rm bdd}= \CohR{X}.
	\]
In particular, if $X$ is definably compact then $\InvR{X} \cap \beta_X(\bar \R)= \CohR{X}$ and we have a canonical continuous retraction $r|_{\InvRt{X}}:\InvRt{X} \rightarrow \CohR{X}$.	
\end{theosin}
\noindent Thus, the spectral topology identifies finitely satisfiable types inside the invariant types. It is worth noticing that the above results are based on the following result that is interesting by itself: \emph{given a definable subset $Z\subset \bar{R}^n$ such that $Z\cap R^n\neq \emptyset$, there is an open definable subset $U\subset \bar{R}^n$ such that $Z\subset U$ and $Z\cap R^n=U\cap R^n$} (\emph{i.e.} external definable subsets of $R^n$ can be defined via open subsets of $\bar{R}^n$). Another key ingredient in the proof of the above result is the characterization of forking in o-minimal structures provided by Dolich \cite{aD04}.

The second objective of the paper is to use the intuition gained  from the above theorem to analyse certain objects in abstract NIP contexts. Recall that a complete first order $\LL$-theory has the \emph{non independence property} (NIP) if there is no $\LL$-formula $\phi(x,y)$ and there is no model $\mathcal{M}$ of the theory for which there are sequences $(a_i)_{i\in\N}$ and $(b_I)_{I\subset \N}$ such that $\phi(a_i,b_I)$ holds if and only if $i\in I$. NIP theories include stable theories (as the theory of algebraically closed fields), o-minimal theories (as the theory of real closed fields) or $P$-minimal theories (as the theory of the $p$-adic field). 

Let $T$ be a complete NIP $\LL$-theory, let $\bar\M$ be a saturated elementary extension of a model $\M$ of $T$ and let $X\subset M^n$ be an $M$-definable subset. In \cite{pS15}, Simon constructs a retraction $$F_M: \Inv{X}\rightarrow \Coh{X},$$
which is continuous with respect to the Stone topology. As the author claims: 
\begin{quote}
	\emph{This map remains rather puzzling. It would be nice to have a better understanding of it, for instance a different construction leading to it.}
\end{quote}
Despite of the theorem above, the retractions $r|_{\InvRt{X}}$ and $F_M$ are completely unrelated (see Example \ref{contraejemplo}).
Nevertheless, since a canonical retraction exist under the presence of a normal spectral topology, it seems natural to ask: 

\emph{Can one endow $\Inv{X}$ with a topology so that it becomes a normal spectral space such that the closed points of this topology are the finitely satisfiable types and the natural retraction from $\Inv{X}$ to the closed points $\Coh{X}$ is exactly $F_M$?}

We give a positive answer to this question by considering a new topology on $\Inv{X}$, that we call the \emph{honest topology}. Recall that an $\bar{M}$-definable subset $Z\subset \bar{M}^n$ is \emph{honest} if for every $M$-definable subset $X\subset M^n$ with $Z\cap M^n \subset X$ we have $Z\subset X(\bar{M})$, where $X(\bar{M})$ denotes the set defined in $\bar\M$ by the formula that defines $X$. Honest sets were introduced in \cite{CS13}, where it is proven assuming NIP that for every $\bar M$-definable subset $Y\subset \bar M^n$, there is an $\bar M$-definable honest subset $Z\subset \bar M^n$ such that $Z\cap M^n=Y\cap M^n$.

The set of invariant types with honest topology is denoted by $\Invh{X}$ and a base of closed sets is given by the sets $[Z]$ for $Z$ an honest set. We show that:

\begin{theosin}[Theorem \ref{thm:normal}] Let $T$ be a complete
	NIP theory. The set $\Invh{X}$ is a normal spectral topological space whose closed points are exactly the finitely satisfiable types, and the natural retraction onto the closed points $$r_M^{\rm h}:\Invh{X} \rightarrow \Coh{X}$$ coincides with $F_M$.	
\end{theosin}	

 We finish this introduction with the organization of the paper. Section \ref{sec:spectral} serves mainly as a preliminary section. In Subsection \ref{section:spectral} we summarize all the basic definitions and results concerning spectral spaces regarding $S^{\rm t}_X(\M)$. As far as possible, we work in a topological first order structure $\M$, which includes for example the theory of the $p$-adic field (and not only the o-minimal setting). All results are well-known (as well as the results of Subsection \ref{subsection:closedtypes}), but in some cases we provide alternative proofs with a model-theoretic flavor instead of using the language of spectral spaces. In  Subsection \ref{sectio:universal} we give a model-theoretic characterization of the closed points $\beta_X(\M)$ of the spectral space $S^{\rm t}_X(\M)$ via a universal property. In Subsection \ref{subsection:closedtypes} we compute a bound of the cardinal of the closure $\textrm{cl}^{\rm t}(p)$ of a type $p\in S^{\rm t}_X(\M)$ in terms of $\dim(p)$.
 
 In Section  \ref{section:relation} we relate the o-minimal spectrum with two natural objects.  Even though we already know that $S^{\rm}_X(\R)$ is not homeomorphic to the real spectrum of the ring of continuos functions, in Subsection \ref{realspectrum} we show that their closed points are naturally homeomorphic. In Subsection \ref{subsec:inftypes} we show the already explained relation between $S_G^{\rm t}(\R)$ and $S^\mu_G(\mathcal{R})$ (the first theorem of this introduction).
 
 In Section \ref{sec:invcoh} we handle the study of invariant and coheir types. In Subsection \ref{subsec:cohclosed} we show that $\CohR{X}$ is closed in $\InvR{X}$ and that the induced spectral topology in $\CohR{X}$ coincides with the Stone one. In Subsection \ref{subsec:inv} we prove that if $X$ is closed then the invariant bounded types that are closed with respect to the spectral topology are exactly the finitely satisfiable types (the second theorem of this introduction).

In Section \ref{sec: honesty} we analyze Simon's retraction $F_M$.  We introduce the honest topology in the set of invariant types and we show that the space $\Inv{X}$ is a normal spectral space whose set of closed points is  $\Coh{X}$ (the last theorem of this introduction).

\section{The spectral topology}\label{sec:spectral}
In this section we equip the space of types of an o-minimal expansion of a real closed field with the so-called {\em spectral topology} as it is done in \cite{aP88b}. We isolate those topological properties of expansions of real closed fields that are needed and work with topological structures with rather general considerations. We recall the notion of topological structure introduced in \cite{aP87}.

\begin{definition}
A first-order structure $\mathcal{M}$ in a language $\mathcal L$ is a {\em topological structure} if there is a formula $\theta(x,\bar y)$ in $\mathcal L$, with $x$ a single variable, such that the collection $\{\theta(M,\bar b)\}_{\bar b\in M^{|y|}}$ forms a basis of a $\mathrm{T}_1$ topology without isolated points.   
\end{definition}



Let $\mathcal{M}$ denote a topological first-order $\LL$-structure and note that if $X$ is a definable set, then so is the closure $\mathrm{cl}_X(V)$ of $V$ in $X$ for every definable subset $V$ of $X$. 
Assume further that $\mathcal M$ satisfies the following two mild conditions. We will justify in Subsection \ref{sec:spectral} our choice.
\begin{enumerate}
    \item[(A1)] For every $n$, every definable subset of $M^n$ is a Boolean combination of  closed definable subsets of $M^n$.
    
    \item[(A2)] It is {\em completely definably normal}, that is, for any disjoint definable closed subsets $X_1$ and $X_2$ of a definable subset $X\subset M^n$, there are disjoint open definable neighborhoods $U_1$ and $U_2$ in $X$ of $X_1$ and $X_2$ respectively.  
\end{enumerate}

In \cite[Cor.\,4.2]{SW19} the authors prove that certain structures of dp-minimal theories which are topological structures satisfy (A1). These include for example dp-minimal expansions of divisible ordered abelian groups and dp-minimal expansions of valued fields. For example, o-minimal and  weakly o-minimal theories expanding an ordered group, or P-minimal theories expanding a ($p$-adically closed) field.

On the other hand, it is well-known that any o-minimal expansion  $\mathcal{M}$ of a real closed field satisfies (A2). Indeed, note that, given a definable set $Z\subset M^n$ in an o-minimal expansion $\M$ of a real closed field, the distance map $\mathrm{dist}(x,Z)$ is definable, see \cite[Lem.\,6.3.5]{vdD98}. If $\mathcal{M}$ is a valued field, then we can use the valuation in a similar way in order to prove (A2). Hence, if $\mathcal M$ is the $p$-adic field or it is an o-minimal expansion of a real closed field, then $\M$ satisfies both (A1) and (A2).

Regarding weakly o-minimal expansions of ordered groups, to the best of our knowledge it is not known if they satisfy (A2):

\begin{quest} \emph{Is every weakly o-minimal expansion of an ordered group definably normal?}  \end{quest}	

\noindent We point out that if $\mathcal{M}$ is a weakly o-minimal structure expanding an ordered group then two disjoint closed subsets of $M^1$ can be easily separated by definable open sets.

\subsection{Spectral topology}\label{section:spectral} Given a definable set $X\subset M^n$, we denote as usual $S_X(\M)=S_n(\M)\cap [X]$ the set of $n$-types over $M$ which concentrate on $X$. This set of types is naturally equipped with the Stone topology, a quasi-compact Hausdorff totally disconnected topology. However, observe that the topology on $\mathcal M$ does not play any role in the definition of the Stone topology. Henceforth we will identify a definable set with the formula that defines it and therefore an element of $S_X(\M)$ we will be written as an ultrafilter of sets, or as a maximal consistent collection of formulas.

We now consider a topology on the set $S_X(M)$ coarser than the Stone topology, which was introduced in \cite{aP88b}. A basic open subset in this new topology is a set $[\phi]$ where $\phi(\M)$ is an open subset of $X$. The set $S_X(\M)$ endowed with this topology will be denoted by $\Spec_X(\M)$ and the closure of a set $C$ by $\mathrm{cl}^{\rm t}(C)$. We point out that $\Spec_X(\M)$ is a \emph{spectral space}, a notion that we recall next (see \cite{DST19}).

\begin{definition}\label{def:spectralspace}
A {\em spectral space} is a topological space $S$ that satisfies the following conditions:
\begin{enumerate}
    \item[(S1)] It is quasi-compact and $\mathrm{T}_0$.
    \item[(S2)] The set of all quasi-compact open subsets forms a basis of open sets.
    \item[(S3)] The intersection of two quasi-compact open subsets of $S$ is again quasi-compact.  
    \item[(S4)] It is \emph{sober}, that is, every nonempty closed and irreducible subset of $S$ is the closure of a unique point.
\end{enumerate}
The topology on $S$ is called the \emph{spectral topology}.
\end{definition}
We include a short proof for the sake of completeness of the following result.

\begin{fact}\cite[Lem.\,1.1]{aP88b}\label{F:Spectral} The topological space $\Spec_X(\M)$ is spectral.
\end{fact}
\begin{proof}
It is clear that the map $S_X(\M) \to \Spec_X(\M)$ given by $p\mapsto p$ is continuous and thus $\Spec_X(\M)$ is quasi-compact. Also, it is $\mathrm{T}_0$ by (A1) because any two distinct types differ in a closed set. Moreover, the continuity of the map above yields that every basic open set is quasi-compact. Hence, the set of all quasi-compact open sets is precisely the set of all basic open sets, by model theoretic compactness. Thus (S2) and (S3) follow. Finally, to show that $\Spec_X(\M)$ is sober, fix a nonempty closed and irreducible subset $C$ of $\Spec_X(\M)$. The set
\[
\Sigma(x)=\{ \phi(x)\in \LL_M \ | \ C\subset [\phi],\  [\phi] \text{ closed} \} \cup \{  \neg\phi(x)\in \LL_M \ | \ C\not\subset [\phi],\ [\phi] \text{ closed}\}
\]
is consistent and in fact it determines a unique complete type by (A1). If $p$ denotes this unique completion, we conclude that $C=\mathrm{cl}^{\rm t}(p)$. Consequently (S4) also holds because any type $q$ satisfying $C=\mathrm{cl}^{\rm t}(q)$ coincides with the completion $p$ of $\Sigma$. \end{proof}

\begin{remark}1) Condition (A1) is equivalent to the fact that $\Spec_X(\M)$ is $\mathrm{T}_0$. 
	
\noindent 2) If a subset $\widetilde V$ of $S^{\rm t}_X(\M)$ is open and quasi-compact, then $\widetilde V=[V]$ for some open definable subset $V$ of $X$. In particular, if $\widetilde C\subset S^{\rm t}_X(\M)$ satisfies that $S^{\rm t}_X(\M)\setminus \widetilde C$ is open and quasi-compact, then $\widetilde C=[C]$ for a closed definable subset $C\subset X$. Thus, the collection of sets of the form $[U]$ and $[C]$ with $U\subset X$ definable open and $C\subset X$ definable closed constitutes a subbasis of the constructible topology of $S^{\rm t}_X(\M)$ (see \cite[Def.\,1.3.11]{DST19}). So by (A1), the sets of the form $[Y]$ for $Y$ a definable subset of $X$ is a basis of the constructible topology. Consequently, the constructible topology of $S^{\rm t}_X(\M)$ is the Stone topology of $S_X(\M)$.
\end{remark}

As we already noticed in the Introduction,  spectral spaces have played an important role in real algebraic geometry. In \cite{CR82} the authors introduce the real spectra of a ring and prove that it is a normal spectral space with good properties. Assumption (A2) implies the same properties  in our context:

\begin{lemma}\label{lemanormal}
The space $\Spec_X(\M)$ is normal. Moreover, for all $p,p_1,p_2\in \Spec_X(\M)$ with $p_1,p_2\in \mathrm{cl}^{\rm t}(p)$, either $p_1\in \mathrm{cl}^{\rm t}(p_2)$ or $p_2\in \mathrm{cl}^{\rm t}(p_1)$.
\end{lemma}
\begin{proof}
Indeed, let $C_1$ and $C_2$ be two disjoint closed subsets of $\Spec_X(\mathcal{M})$. Then there are families of closed definable subsets $\{X_i\}_{i\in I}$ and \{$Y_j\}_{j\in J}$ such that $C_1=\bigcap_i [X_i]$ and $C_2=\bigcap_j [Y_j]$. By quasi-compactness there are two disjoint closed definable subsets $X_{i_0}$ and $Y_{j_0}$ such that $C_1\subset [X_{i_0}]$ and $C_2\subset [Y_{j_0}]$. As $\M$ is definably normal (condition (A2)), we conclude that there are disjoint definable open subsets $U$ and $V$ with $X_{i_0}\subset U$ and $Y_{j_0}\subset V$, so $C_1\subset [U]$ and $C_2\subset [V]$ with $[U]\cap [V]=\emptyset$, as required.

For the second part, let $p_1,p_2\in \mathrm{cl}^{\rm t}(p)$ and suppose that $p_1\notin \mathrm{cl}^{\rm t}(p_2)$ and $p_2\notin \mathrm{cl}^{\rm t}(p_1)$. Then there are two closed definable subsets $C_1$ and $C_2$ of $X$ such that $C'_1:=C_1\setminus C_2\in p_1$ and $C'_2:=C_2\setminus C_1\in p_2$. Consider the open definable set $U:=X\setminus (C_1\cap C_2)$ of $X$. Since $U$ is definably normal and both $C'_1$ and $C'_2$ are closed in $U$, there are open definable sets $V_1$ and $V_2$ of $U$ such that $V_1\cap V_2=\emptyset$ with $C'_1\subset V_1$ and $C'_2\subset V_2$. In particular, since $p_i\in [V_i]$ for each $i=1,2$, we obtain $p\in [V_1]\cap [V_2]$, so $V_1\cap V_2\neq \emptyset$, a contradiction.\end{proof}

Another milestone in real algebraic geometry is the paper of Carral and Coste \cite{CC83}, where they compare the cohomological and Krull dimensions of a normal spectral space. In their analysis the set of closed points plays an important role, because the cohomological dimension of a normal spectral space coincides with the cohomological dimension of its closed points. 
 Recall that in a topological space $S$, given $x,y\in S$, we say that $x$ is a {\em specialization} of $y$ if $y\in \overline{\{x\}}$.
 Usually, this is denoted by $x \rightsquigarrow y$. Note that the arrow $\rightsquigarrow$ induces a partial order, and that with this notation Lemma \ref{lemanormal} says that for every $p\in S^t_X(\M)$ the set $\mathrm{cl}^{\rm t}(p)$ is totally ordered under specialization. We denote by $\text{Max}(S)$ the set of \emph{maximal points} of $S$ with respect to the specialization order. Clearly, the maximal points are  the closed points of $S$. In the context of normal spectral spaces, the maximal points have the following crucial property:

\begin{fact}\label{fact:max}Let $S$ be a normal spectral space. Then $\text{Max}(S)$ is a quasi-compact Hausdorff space. Moreover, for each $x\in S$ there is a unique $r(x)\in \text{Max}(S)$ with $r(x)\in \text{cl}(x)$, and the retraction map $S\rightarrow \text{Max}(S):x\mapsto r(x)$ is continuous, closed and proper.
\end{fact}
\begin{proof}See  \cite[Prop.\,3]{CC83}  and \cite[Prop.\,4.1.2]{DST19}.
\end{proof}
If $X$ is a semialgebraic set, then the set of maximal points of the spectrum of the ring of semialgebraic functions on $X$ (see Subsection \ref{realspectrum}) is traditionally denoted with the letter $\beta$ because it is related with the Stone-\v{C}ech compactification of $X$, see \cite{FG12}. Since our context is similar to the latter, given a definable set $X$ of $\mathcal{M}$ we will denote by $\beta_X(\mathcal{M})$  the set of closed points of $\Spec_X(\mathcal{M})$, that is,
\[
\beta_X(\mathcal{M}) =\{ p\in \Spec_X(\mathcal{M}) \ | \ \{ p \} \text{ is closed}\} 
.\] Note that $X$ is included in $\beta_X(\mathcal M)$ by considering the map $a\mapsto \tp(a/M)$ for $a$ in $X$. Moreover, the induced topology from $M^n$ on $X$ coincides with the one induced from $\beta_X(\M)$ on $X$. The following is a straightforward  consequence of Lemma \ref{lemanormal} and Fact \ref{fact:max}. 
\begin{cor}\label{retractionspec}{\em (1)} The space $\beta_X(\mathcal M)$ is quasi-compact and Haussdorf.
	
\noindent {\em (2)} There is a continuous closed proper retraction $r:\Spec_X(\mathcal{M})\rightarrow \beta_X(\mathcal{M})$ that maps $p\in\Spec_X(\mathcal{M})$ to the unique $m\in \beta_X(\mathcal{M})$ such that $m\in \mathrm{cl}^{\rm t}(p)$.
\end{cor}
\begin{proof}We give a  proof specialized to our context for the sake of the reader. Let us show first that for every $p\in \Spec_X(\M)$ there is $m\in \beta_X(\mathcal{M})$ such that $m\in \text{cl}^{\rm t}(p)$. Since $\Spec_X(\M)$ is quasi-compact, by Zorn's lemma the set $\{\text{cl}^{\rm t}(y):y\in \text{cl}^{\rm t}(p)\}$ has a minimal element $\text{cl}^t(m)$ with respect to the inclusion partial order. The space $\Spec_X(\M)$ is $\mathrm{T}_0$ and therefore $m\in \beta_X(\mathcal{M})$, as required. Moreover, by Lemma \ref{lemanormal} such a type $m\in \beta_X(\mathcal{M})\cap \text{cl}^t(p)$ is the unique closed point in $\text{cl}^{\rm t}(p)$. Indeed, if $m,m'\in \mathrm{cl}^{\rm t}(p)\cap \beta_X(\M)$ were distinct, then there would exist two disjoint respective open sets $\widetilde U$ and $\widetilde V$ of $m$ and $m'$ with $p\in \widetilde U\cap \widetilde V$, which is a contradiction. 

Once we have established that for every $p\in \Spec_X(\M)$ there is a unique specialization $m\in \beta_X(\mathcal{M})$ such that $m\in \mathrm{cl}^{\rm t}(p)$, we define
\begin{align*}
    r: \Spec_X(\mathcal{M}) & \rightarrow \beta_X(\mathcal{M}) \\ p & \mapsto r(p)=m.
\end{align*}
It remains to check that the  retraction $r$ is continuous and closed. To show that $r$ is continuous, we fix a definable closed subset $C$ of $X$. We first claim that:
\begin{equation}
    r^{-1}([C])=\bigcap \left\{ [V]  \ | \ C\subset V\subset X \text{ open and definable}\right\}.
\end{equation}
Fix $p\in r^{-1}([C])$  and suppose there is an open definable subset $V$ of $X$ containing $C$ such that $p\notin[V]$. Since $r(p)\in \mathrm{cl}^{\rm t}(p)$, we deduce $r(p)\notin [V]$, so $r(p)\notin [C]$, which is a contradiction. To prove the other inclusion, suppose that $p\in [V]$ for every open definable subset $V$ of $X$ containing $C$, and assume that $p\notin r^{-1}([C])$. Then $r(p)\notin [C]$, so $r(p)\in [X\setminus C]$. Since $X\setminus C$ is open and $r(p)\in \mathrm{cl}^{\rm t}(p)$, we deduce that $p\in [X\setminus C]$, so $p\notin [C]$. Thus, we have 
\begin{equation}
[C]\subset \bigcup \left\{ [X\setminus F] \ | \ p\in [F] \text{ and } F\subset X \text{ definable and closed} \right\}.    
\end{equation}
Indeed, suppose there is $q\in [C]$ such that $q\in [F]$ for every closed definable subset $F$ of $X$ with $p\in[F]$. Thus $q\in \mathrm{cl}^{\rm t}(p)$, so $r(q)=r(p)$. Since $q\in [C]$ and $C$ is closed, $r(q)\in [C]$ and hence $r(p)\in [C]$, which is a contradiction. Therefore, we obtain (2). Now, since $[C]$ is quasi-compact in the Stone topology we deduce that $[C]$ is contained in $[X\setminus F]$ for some closed definable subset $F$ of $X$ with $ p\in [F]$. Thus,  since $p\in [V]$ for every open definable subset $V$ of $X$ containing $C$, and $X\setminus F$ is an open definable subset with $C\subset X\setminus F$, the choice of $p$ yields that $p\in [X\setminus F]$, which is a contradiction. Therefore, we obtain (1).

Finally, since $\mathcal{M}$ is definably normal, we deduce that 
\begin{equation}
r^{-1}([C])=\bigcap \left\{ [\mathrm{cl}_X(V)] \ | \ C\subset V\subset X \text{ open and definable}\right\},
\end{equation} 
where $\mathrm{cl}_X(V)$ denotes the topological closure of $V\subset X$ in $X$. Indeed, by (1) it is enough to show that if $p\in [\mathrm{cl}_X(V)]$ for every open definable subset $V$ of $X$ containing $C$, then $p\in [V]$ for any such $V$. Fix an  open definable subset $V$ of $X$ containing $C$ and note that there exists an open definable subset $U$ such that $C\subset U\subset \mathrm{cl}_X(U)\subset V$. It is enough to take two disjoint open definable sets $U$ and $W$ with $C\subset U$ and $V^c \subset W$, which exist as $\M$ is definably normal (by the (A2) assumption). Hence,  we have that $p\in [\mathrm{cl}_X(U)]$ as $U$ is open and definable, so $p\in [V]$, as desired. Therefore, we obtain (3) and $r$ is continuous. 
 
Once we have showed that $r:\Spec_X(\mathcal{M})\rightarrow \beta_X(\mathcal{M})$ is continuous, let us remark that it is also closed and proper because $\Spec_X(\mathcal{M})$ is quasi-compact and $\beta_X(\mathcal{M})$ is Hausdorff.
\end{proof}

The set of closed points contains a set which is important from a model-theoretic point if view. 

\begin{lemma}\label{L:BetaClosed}
Let $X\subset M^n$ be definable. The set $\beta_X(\M)$ contains the set of all realized types, which is dense in $\Spec_X(\M)$.
\end{lemma}
\begin{proof}
Since $\M$ is equipped with a $\mathrm{T}_1$ topology, the definable set $\{m\}$ is closed for each $m\in X$. Thus, the set $[x=m]\subset \Spec_X(\M)$ is also closed. As $[x=m]$ only contains the type $\tp(m/M)$, realized types belong to $\beta_X(\M)$. Furthermore, the set of realized types is dense since any consistent $\LL_M$-formula is realized in $\M$.
\end{proof}

We point out that in general $\beta_X(\M)$ does not have a good behavior under elementary extensions nor intersections. To illustrate this we consider the following toy example.  
\begin{remark}\label{R:BetaNotGood}
Let $\R$ be the field structure of the real numbers and let $X=\mathbb R$. It is immediate to verify that
\[
\beta_X(\R)=\left\{ \tp(a/\mathbb R) \right\}_{a\in \mathbb R}\cup\{ p_{-},p_{+}\},
\]
where $p_{-}:=\{x<a\ | \ a\in\mathbb R\}$ and $p_{+}:=\{x>a\ | \ a\in\mathbb R\}$. We remark:
\begin{enumerate}[(i)]
    \item The space $\beta_X(\R)$ does not behave well under elementary extensions nor restrictions. Indeed,  if  $\mathcal R_1$ is an elementary extension of $\R$ and $\varepsilon\in R_1$ is a positive infinitesimal, then $\tp(\varepsilon/R_1)\in \beta_X(\mathcal R_1)$ whereas $\tp(\varepsilon/\mathbb{R})\notin \beta_X(\mathcal{R})$.
    \item Let $Y=(0,\infty) \subset X$. Then $\beta_Y(\R) \neq \beta_X(\R)\cap [Y]$.  Indeed,  using the notation from (i), we have 
    \[
\beta_Y(\R)=\left\{ \tp(a/\mathbb R) \right\}_{a\in \mathbb R_{>0}}\cup\{ \tp(\varepsilon/\mathbb R),p_{+}\}.
\]
    Hence, the type $\tp(\varepsilon/\mathbb{R})\in \beta_Y(\R)$ but it does not belong to $\beta_X(\R)\cap [Y]$. 
    
    \item It also follows that $\beta_X(\R)$ is not closed with respect to the spectral topology. For example, the type $\tp(\varepsilon/\mathbb R)\in \mathrm{cl}^{\rm t}(\beta_X(\R))$. In this case, the subset  $\beta_X(\R)$ is closed if and only if $X$ is finite.
\end{enumerate}
\end{remark}

\subsection{A universal property for closed points}\label{sectio:universal}

Let $\mathcal{M}$ be a topological structure as above satisfying (A1) and (A2) and let $X\subset M^n$ be a definable set. We consider the subspace topology on $X$ and we establish  a universal property for $\beta_X(\mathcal{M})$ in the vein of \cite[Lem.\,3.7]{GPP14}.

\begin{definition}Let  $C$ be a topological space. A continuous map $f:X\rightarrow C$ is {\em definably separated by closed sets} if for every two disjoint closed subsets $C_1$ and $C_2$ of $C$ there are two disjoint closed definable subsets $X_1$ and $X_2$ of $X$ such that $f^{-1}(C_1)\subset X_1$ and $f^{-1}(C_2)\subset X_2$.	
\end{definition}	

\begin{remark}\label{R:DefSep} The map $i: X\hookrightarrow \beta_X(\mathcal{M})$ is clearly definably separated by closed sets. Indeed, given two disjoint closed subsets $C_1$ and $C_2$ of $\beta_X(\M)$, by quasi-compactness there are two closed definable subsets $V_1$ and $V_2$ of $X$ such that each $C_j \subset \beta_X(\M)\cap [V_j]$ and 
\[
\left( \beta_X(\M) \cap [V_1] \right) \cap \left( \beta_X(\M) \cap [V_2] \right) =\emptyset.
\] 
Since each $i^{-1}\left( \beta_X(\M) \cap [V_j] \right))=V_j$, it follows immediately that $V_1\cap V_2=\emptyset$ and $i^{-1}(C_j) \subset V_j$, as required.
\end{remark}
It is easy to show that given a continuous map $f:X\rightarrow C$ from $X$ to a topological space $C$, if there exist a continuous extension $\hat f:\beta_X(\mathcal{M})\rightarrow C$ of $f$, then $f$ is definably separated by closed sets, by the remark above. We show the converse:

\begin{prop}Let $f:X\rightarrow C$ be a continuous map, where $C$ is quasi-compact and Hausdorff. If $f$ is a definably separated by closed sets, then there is a continuous extension $\hat f:\beta_X(\mathcal{M})\rightarrow C$ of $f$ to $\beta_X(\mathcal{M})$.
\end{prop}
\begin{proof}We first define $\hat f$ as follows: Given a type $p\in \beta_X(\mathcal{M})$, set
$$
\hat f(p):=\bigcap_{U\in p \text{ open}} \overline{f(U)}.$$
The intersection above is not empty, (by quasi-compactness) because it has the finite intersection property. Thus, it suffices to show that the above intersection contains a single point. Suppose otherwise there are $a$ and $b$ in $\hat f(p)$ with $a\neq b$. Since $C$ is regular, there are open respective neighborhoods $U_a$ and $U_b$ of $a$ and $b$ such that $\overline{U}_a\cap \overline{U}_b=\emptyset$. As $f:X\to C$ is a definably separated by closed sets, there exist two disjoint closed definable subsets $X_a$ and $X_b$ of $X$ such that $f^{-1}(\overline{U}_a)\subset X_a$ and $f^{-1}(\overline{U}_b)\subset X_b$. We can assume that $p\notin [X_a]$, so that $p\in [X\setminus X_a]$. Since $X\setminus X_a$ is open, $a\in \hat f(p)\subset \overline{f(X\setminus X_a)}$. Hence, as $a\in U_a$ and $a\in \overline{f(X\setminus X_a)}$, it follows that $U_a\cap f(X\setminus X_a)\neq \emptyset$. In particular, $f^{-1}(\overline{U}_a)\cap (X\setminus X_a)\neq \emptyset$, which is a contradiction as $f^{-1}(\overline{U}_a)\subset X_a$.

Once we have proved that $\hat f$ is a function, we prove that it is  continuous. Let $Z$ be a closed subset of $C$ and suppose that  there is some $p\in \overline{\hat f^{-1}(Z)}\setminus \hat f^{-1}(Z)$. Since $\hat f(p)\notin Z$, we have 
$$
Z\cap \bigcap_{U\in p \text{ open}} \overline{f(U)}=\emptyset,
$$
so $Z\cap \overline{f(U)}=\emptyset$ for some definable open set $U$ with $p\in [U]$ (by quasi-compactness). Since $[U]$ is open and $p\in \overline{\hat f^{-1}(Z)}$, it follows that $[U]\cap \hat f^{-1}(Z)\neq \emptyset$, so there is some $q\in[U]$ such that $\hat f(q)\in Z$. By definition we also have $\hat f(q)\in \overline{f(U)}$, hence $\hat f(q)\in  Z\cap \overline{f(U)}$, which is a contradiction.
\end{proof}	

\begin{remark}\label{jexpansion}Let $X$ be definable in $\M$ and let $\mathcal{M}'$ be an expansion of $\mathcal{M}$ satisfying (A1) and (A2). Then we can consider also $\beta_X(\mathcal{M}')$. It turns out that there is a closed continuous surjective map
$$j:\beta_X(\mathcal{M}')\rightarrow \beta_X(\mathcal{M}).$$
Indeed, the inclusion continuous map  $X\hookrightarrow \beta_X(\mathcal{M})$ is definably separated by closed sets with respect to $\mathcal{M}'$ and therefore there exists a continuous extension $j$. The map $j$ is closed and therefore it is  surjective because $X$ is dense in both $\beta_X(\mathcal{M})$ and $\beta_X(\mathcal{M}')$.

Let us describe the map $j$. Given $p\in S_X(\mathcal{M}')$, denote by $p^\mathcal{M}$ the restriction of $p$ to the language of $\M$, and let $r:\Spec_X(\mathcal{M})\rightarrow \beta_X(\mathcal{M})$ be the natural retraction provided in Corollary \ref{retractionspec}. The map 
\[
\beta_X(\mathcal{M}')  \rightarrow \beta_X(\mathcal{M}) , \  
p \mapsto r(p^\mathcal{M})
\]
is continuous and coincide with $j$ on realized types, so they are equal.
\end{remark}

The following example shows that the map $j$ from Remark \ref{jexpansion} is not a homeomorphism in general. 
\begin{example}
Let $\mathcal{R}$ be the field of real numbers and $\mathcal{R}_{\rm exp}$ be the field of real numbers with the exponentiation. Set $X=\mathbb R^2$ and consider in $\beta_X(\R_{\rm exp})$ the types
\begin{align*}
	q_1 & :=\{Y\subset \mathbb{R}^2 \ | \ Y \text{ definable in }\mathcal{R}_{\rm exp} \ \& \ \exists t_0>0,\forall \, t>t_0, (t,0)\in Y \}, \\
	q_2 & :=\{Y\subset \mathbb{R}^2 \ |  \ Y \text{ definable in }\mathcal{R}_{\rm exp} \ \& \ \exists t_0>0,\forall \, t>t_0, (t,e^{-t})\in Y \}.
\end{align*}
In this case $q_1^\mathcal{R}\in \beta_X(\mathcal{R})$ but $q_2^\mathcal{R}\notin \beta_X(\mathcal{R})$. Moreover, $q_1^\mathcal{R}\in \mathrm{cl}^{\rm t}(q_2^\mathcal{R})$. Indeed, given an open semialgebraic set $U$ of $\mathbb{R}^2$ such that $(t,0)\in U$ for every $t$ large enough, we have $(t,e^{-t})\in U$ for $t$ large enough. Thus,
\[
j(q_1)=r(q_1^\mathcal{R})=r(q_2^\mathcal{R})=j(q_2).
\]
\end{example}

\subsection{Closed types in topological structures with a dimension}\label{subsection:closedtypes}
Let $\mathcal{M}$ be a topological structure as above satisfying (A1) and (A2). Assume further that $\M$ is equipped with a dimension satisfying that for all definable sets $V\subset X\subset M^n$ we have that $\dim(\mathrm{cl}_X(V)\setminus V)<\dim(V)$. This is the case for example if $\mathcal{M}$ is an o-minimal expansion of a real closed field or the $p$-adic field.

\begin{remark}The existence of such a well-behaved dimension map implies assumption (A1). Indeed, note first that assumption (A1) is equivalent to the fact that for any two different types there is an open (equivalently, closed) set which contains one of these points and not the other. This follows by a standard topological argument \cite[Ex.\,3.1.1]{TZ12}. Now, let $p$ and $q$ be two distinct types in $S_X(\M)$. Let $V$ be a definable set which distinguish them, that is to say that $V$ is in one type but not the other. Clearly, we can take $V$ so that $\dim(V)$ is minimal with this property. Assume that $V$ is not closed (as otherwise we are done) and suppose that $V\in p$ but $V\notin q$. Since $\dim(\mathrm{cl}_X(V)\setminus V)<\dim(V)$ and $(\mathrm{cl}_X(V)\setminus V)\notin p$, we have that $(\mathrm{cl}_X(V)\setminus V)\notin q$ as well due to the minimality of the dimension of $V$. Thus, $\mathrm{cl}_X(V)\notin q$ and $\mathrm{cl}_X(V)\in p$, so $p$ and $q$ are distinguished by the closed set $\mathrm{cl}_X(V)$.  
\end{remark}

Let $X\subset M^n$ be a definable set in $\M$. By Corollary \ref{retractionspec} we know that the space of closed types $\beta_X(\M)$ is quasi-compact and Hausdorff and there is a natural retraction $r: \Spec_X(M)\rightarrow \beta_X(\M)$. The following lemma exhibits the relation between dimension and the spectral topology, where the dimension $\dim(p)$ of the type $p$ is defined as $\min\{\dim(X) \ | \ X\in p\}$.

\begin{lemma}\label{L:DimBeta} Let $p,q\in S_X(\M)$. The following holds:
\begin{enumerate}
    \item If $q\in \mathrm{cl}^{\rm t}(p)$ and $q\neq p$, then $\dim(q)< \dim(p)$. 
    \item $\dim(r(p))\le \dim(p)$, with equality if and only if $p\in \beta_X(\M)$.
\end{enumerate}
In particular, for every $p\in S_X(\M)$ the  number of elements in $\mathrm{cl}^{\rm t}(p)$ is less than or equal to $1+\dim(p)$.	
\end{lemma}
\begin{proof}
(1) Let $W\subset X$ be a definable set in $p$ of minimal dimension. Since $q\in \mathrm{cl}^{\rm t}(p)$, we have $\mathrm{cl}_X(W)\in p$ and $\mathrm{cl}_X(W)\in q$, so  we can assume $\dim(p)=\dim(X)$ after substituting $X$ by $\mathrm{cl}_X(W)$. 

Since the spectral topology is $\mathrm{T}_0$ and $p\neq q$, it follows $p\notin \mathrm{cl}^{\rm t}(q)$. Therefore, there is a definable open subset $U$ of $X$ such that $p\in [U]$ but $q\notin [U]$. Since $p\in [\mathrm{cl}_X(U)]$, we deduce  $q\in [\mathrm{cl}_X(U)\setminus U]$, so \[\dim(q)\leq \dim(\mathrm{cl}_X(U)\setminus U)<\dim (\mathrm{cl}_X(U))=\dim(U)=\dim(X)=\dim(p),\] as required.

(2) follows immediately from (1). The bound of the cardinality of $\mathrm{cl}^{\rm t}(p)$ follows from (1) and  Lemma \ref{lemanormal}.
\end{proof}

As we pointed out in Lemma \ref{L:BetaClosed} the space $\beta_X(\M)$ contains all realized types which concentrate on $X$, but it also may contain types of maximal dimension as happens in the example of Remark \ref{R:BetaNotGood}. Nonetheless, these two kinds of closed types may have distinct fibers with respect to the retraction $r$. In fact, the size of the fiber characterizes types of maximal dimension in the o-minimal setting.
\begin{prop}
Let $\R$ be an o-minimal expansion of a real closed field. Let $p\in \beta_{R^n}(\R)$ be any type. Then 
$r^{-1}(\{p\})=\{p\}$ if and only if $\dim(p)=n$. 
   
\end{prop}
\begin{proof}
		Suppose first that there is some $q\neq p$ with $r(q)=p$. Note that $q\not\in \beta_{R^n}(\R)$ and thus $\dim(p)<\dim(q)\leq n$ by Lemma \ref{L:DimBeta}(2).
		
		To prove the other direction, assume that $\dim(p)<n$. We can assume there is an elementary extension $\mathcal N$ of $\mathcal{R}$ and a realization $(\bar a,a)\in N^{n-1}\times N$  of $p$ with $a\in \dcl(R,\bar a)$.  Consider the realized type $p':=\tp(a/N)$, and note there exists some non-realized type $q'=\tp(b/N)$ with $r_\mathcal{N}(q')=p'$, where  $r_\mathcal{N}: \Spec_{N^n}(\mathcal{N}) \rightarrow \beta_{N^n}(\mathcal{N})$. For example, take $b:=a+\varepsilon$ for some infinitesimal $\varepsilon>0$ with respect to $N$. Set $q=\tp(\bar a,b/R)$, which is not $p$ because $b\notin  \dcl(R,\bar a)$.
		
		We claim that $r(q)=p$. To see this, let $U\subset R^n$ be an open $R$-definable set such that $p\in [U]$. We denote by $U(N)\subset N^n$ its realization in $\mathcal{N}$. Consider now the $N$-definable open set
		\[
		W=\{ y \in N \ | \ (\bar a,y)\in U(N) \}
		\]
		and note that $a\in W$. Thus, there is an open interval $V=(c,d)\subset W$ with $c,d\in N$ such that $c<a<d$. It follows that $p'\in [V]$, so $q'\in [V ]$ because $r_{\mathcal{N}}(q')=p'$. Hence, $b\in W$ and so $(\bar a,b)\in U(N)$, which yields that $q\in [U]$. Therefore, as $q\neq p$, we obtain |$r^{-1}(\{p\})|>1$, as required.
\end{proof}

\section{The relation with other topological constructions}\label{section:relation}
In this section we fix an o-minimal expansion $\R$ of a real closed field and we study its connections with its definable real sprectrum, as well as with its space of $\mu$-types introduced by Peterzil and Starchenko \cite{PS17}.

\subsection{The relation with the real spectrum}\label{realspectrum} Let $\mathcal{R}$ be an o-minimal expansion of a real closed field. Let $\mathcal{C}^{\mathcal{R}}(X)$ be the ring of continuous functions from $X$ to $R$ that are definable in $\mathcal{R}$, that is,
\[
\mathcal{C}^{\mathcal{R}}(X) = \left\{ f:X\to R \ | \ f \text{ is continuous and $R$-definable}\right\}.
\] The aim of this section is to prove that $\beta_X(\mathcal{R})$ is homeomorphic to the set of closed points of the real spectrum $\mathrm{Sper}(\mathcal{C}^{\mathcal{R}}(X))$ of $\mathcal{C}^{\mathcal{R}}(X)$. Recall that the set $\mathrm{Sper}(\mathcal{C}^{\mathcal{R}}(X))$ is the collection of \emph{prime cones} of $\mathcal{C}^{\mathcal{R}}(X)$, that is, subsets $\alpha$ of $\mathcal{C}^{\mathcal{R}}(X)$ such that $\mathfrak{p}_\alpha:=\alpha\cap-\alpha$ is a prime ideal of $\mathcal{C}^{\mathcal{R}}(X)$ and $\alpha/\mathfrak{p}_\alpha$ is the positive cone of a total order of $\mathcal{C}^{\mathcal{R}}(X)/\mathfrak{p}_\alpha$ (we refer the reader to \cite[Ch.\,7]{BCR98} for the basic definitions concerning the real spectrum). Let $\rho_\alpha:\mathcal{C}^{\mathcal{R}}(X)\to \mathcal{C}^{\mathcal{R}}(X)/\mathfrak{p}_\alpha$ be the natural homomorphism. The subsets $$D_r(f):=\{\alpha\in\mathrm{Sper}(\mathcal{C}^{\mathcal{R}}(X)) \ | \  \rho_{\alpha}(f)>0\}$$ for $f\in \mathcal{C}^{\mathcal{R}}(X)$, which are called \em basic open subsets\em, constitute a basis of the \emph{spectral topology} of $\mathrm{Sper}(\mathcal{C}^{\mathcal{R}}(X))$.

As we see in the following well-known fact, the ring $\mathcal{C}^{\mathcal{R}}(X)$ has the special feature that the real spectrum is homeomorphic to the (Zariski) spectrum, and therefore we will work with the latter. Before proceeding, let us recall some basic definitions on the (Zariski) spectrum. Let $\mathrm{Spec}(\mathcal{C}^{\mathcal{R}}(X))$ denote the spectrum of $\mathcal{C}^{\mathcal{R}}(X)$, that is, the set of prime ideals of $\mathcal{C}^{\mathcal{R}}(X)$, and we regard it as a topological space with the usual Zariski topology. Recall that given $f\in \mathcal{C}^{\mathcal{R}}(X)$, a basic open set is
\[D(f)=\left\{\mathfrak{p}\in \mathrm{Spec}(\mathcal{C}^{\mathcal{R}}(X)) \ | \ f\notin \mathfrak{p} \right\}.\]

\begin{fact} The support map
	$$
	\mathrm{supp}:\mathrm{Sper}(\mathcal{C}^{\mathcal{R}}(X))\rightarrow \mathrm{Spec}(\mathcal{C}^{\mathcal{R}}(X)), \ \alpha \mapsto \mathfrak{p}_\alpha
	$$
	is a homeomorphism. 
\end{fact}	


\begin{proof}The support map is continuous and its image is the set of real prime ideals of $\mathcal{C}^{\mathcal{R}}(X)$ (see \cite[Prop.\,7.1.8]{BCR98}). Each $\mathfrak{p}\in \mathrm{Spec}(\mathcal{C}^{\mathcal{R}}(X))$ is real prime because the field of fractions of $\mathcal{C}^{\mathcal{R}}(X)/\mathfrak{p}$ is a real field, so $\mathrm{supp}$ is surjective. Now, the field of fractions of $\mathcal{C}^{\mathcal{R}}(X)/\mathfrak{p}$ admits a unique ordering and therefore the map $\mathrm{supp}$ is injective. Indeed, given $f\in \mathcal{C}^{\mathcal{R}}(X)$ we have $(f-|f|)(f+|f|)=0$, so either $f-|f|\in \mathfrak{p}$ or $f+|f|\in \mathfrak{p}$, where $|f|$ denotes the absolute value of $f$. Since $|f|=(\sqrt{|f|})^2$, the ordering of $\mathcal{C}^{\mathcal{R}}(X)/\mathfrak p$ is completely determined by $\mathfrak{p}$.	
	
Finally, we note  that $\text{supp}(D_r(f))=D(f+|f|)$ for every $f\in \mathcal{C}^{\mathcal{R}}(X)$,
which yields that the inverse of the support map is continuous. 
\end{proof}

The real spectrum of a ring is a normal spectral space by \cite[Rmk.\,7.1.17, Prop.\,7.1.25]{BCR98}, and therefore we deduce the following.
\begin{fact}
	The space $\mathrm{Spec}(\mathcal C^\R(X))$ is a normal spectral space.
\end{fact}
\begin{proof}For the sake of the reader, we include the proof of normality without using the machinery of the real spectrum. Let $C_1$ and $C_2$ be two disjoint closed subsets of  $\mathrm{Spec}(\mathcal C^\R(X))$, which we may assume to be basic closed subsets, by quasi-compactness of $\mathrm{Spec}(\mathcal C^\R(X))$. Hence, there are two continuous definable functions $f_1,f_2\in \mathcal C^\R(X)$ such that $C_i=D(f_i)^c$ for $i=1,2$. Consider the closed definable sets  $Z(f_1)=f_1^{-1}(\{0\})$ and $Z(f_2)=f_2^{-1}(\{0\})$. 
	
As $X$ is definably normal, there are two open definable subsets $U_1\supset Z(f_1)$ and $U_2\supset Z(f_2)$ such that $U_1\cap U_2=\emptyset$. Let $g_1,g_2\in \mathcal C^\R(X)$ be the functions defined by $g_i(x)=\text{dist}(x,X\setminus U_i)$ for each $x\in X$ and $i=1,2$. Note that $Z(g_1)=X\setminus U_1$ and $Z(g_2)=X\setminus U_2$.
	 We claim that $C_i\subset D(g_i)$. Otherwise, for $\mathfrak p\in C_i\setminus D(g_i)$ we have that $g_i,f_i\in \mathfrak p$. As 
	\[
	Z(f_i^2+g_i^2)=Z(f_i)\cap Z(g_i)\subset U_i\cap (X\setminus U_i)=\emptyset,
	\]the function $f_i^2+g_i^2\in \mathfrak p$ is a unit, a contradiction. Hence $C_i\subset D(g_i)$ for $i=1,2$. It remains to show that $D(g_1)\cap D(g_2)=\emptyset$. Indeed, since
	\[
	Z(g_1\cdot g_2) = Z(g_1)\cup Z(g_2) = (X\setminus U_1)\cup (X\setminus U_2) = X\setminus (U_1\cap U_2) =X,
	\]
	the function $g_1\cdot g_2$ is the constant function $x\mapsto 0$ and therefore $g_1\cdot g_2\in \mathfrak p$. Hence, as $\mathfrak p$ is prime, either $\mathfrak p\notin D(g_1)$ or $\mathfrak p\notin D(g_2)$.
\end{proof}

Now, consider the map
\[
\iota: \Spec_X(\mathcal{R})  \to \mathrm{Spec}(\mathcal{C}^{\mathcal{R}}(X)) , \ 
    p  \mapsto  \{ f\in \mathcal C^\R(X) \ | \  Z(f)\in p \},
\] 
where $Z(f):=f^{-1}(\{0\})$ is a definable closed set because $f$ is continuous. In fact, any definable closed set $Z$ of $X$ is of this form by considering the map $f:X\to \mathbb R$ given by $f(x)=\mathrm{dist}(x,Z)$.
\begin{prop}\label{P:iota-cont}
The map $\iota: \Spec_X(\mathcal{R})  \to \mathrm{Spec}(\mathcal{C}^{\mathcal{R}}(X)) $ is injective and spectral, that is, it is continuous and the inverse image of a quasi-compact open is quasi-compact.
\end{prop}
\begin{proof}
To prove that it is continuous, consider an arbitrary basic open set $D(f)$ given by some $f\in \mathcal{C}^{\mathcal{R}}(X)$. Then
\[
\iota^{-1}(D(f))=\{p \in \Spec_X(\R) \ | \ \iota(p)\in D(f) \}=\{p  \in \Spec_X(\R) \ | \  Z(f)\notin p\}=[X\setminus Z(f)]
\]
and $\iota$ is continuous because the set $Z(f)$ is a closed subset of $X$. To verify that the inverse image of a quasi-compact open set $U\subset \mathrm{Spec}(\mathcal C^\R(X))$ is quasi-compact, note that $U$ is a finite union of sets of the form $D(f_i)$ for $i\in I$, where $I$ finite and each $f_i\in \mathcal C^\R(X)$. Thus,  
\[
\iota^{-1}(U)=\bigcup_{i\in I} \iota^{-1}(D(f_i)) = \bigcup_{i\in I} [X\setminus Z(f_i)]
\]
is quasi-compact, as so is each set of the form $[V]$ for $V\subset X$ definable since $S_X(\R)\to \Spec_X(\R)$ is continuous. 

Finally, by \cite[Cor.\,2.6 and Prop.\,3.1]{mT99} the map $\iota$ is injective. We provide also a direct proof. Consider $p,q\in \Spec_X(\mathcal{R})$ with $p\neq q$ such that $\iota(p)=\iota(q)$. By assumption (A1), there exists a definable closed subset $Z$ of $X$ such that $Z\in p$ and $Z\notin q$. Since $Z$ is closed, there is some $f\in \mathcal{C}^{\mathcal{R}}(X)$ such that $Z(f)=Z$, so $f\in \iota(p)$ but $f\notin \iota(q)$, a contradiction.
\end{proof}

In general, the map $\iota$ is not surjective even when $\R$ has only the field structure.

\begin{example} 
The following two examples correspond to \cite[Rmk.\,1.2]{FG14} (cf. \cite[Rmk.\,2.6.5]{BCR98}) and \cite[pp. 1]{mT99}, respectively. In both cases, one finds two functions $f,g\in \mathcal C^\R(X)$ with $Z(f)=Z(g)$ and a prime ideal $\mathfrak p\in \mathrm{Spec}(\mathcal C^\R(X))$ containing one but not the other, showing that $\iota$ cannot be surjective.
\begin{enumerate}[(i)]
    \item Let $\R$ be the field of real numbers and $X=\{(x,y)\in \mathbb R\times \mathbb R \ | \ y>0\}\cup\{(0,0)\}$. Consider the semialgebraic functions $f(x,y)=y$ and $g(x,y)=x^2+y^2$. Their zero set agree on $X$ but for each $k\in\mathbb N$ the limit at $(0,0)$ of the semialgebraic function $h_k=g^k/f$ does not exist. Therefore the prime ideal 
    \[\mathfrak p = \left\{ h\in \mathcal C^\R(X) \ \big| \ \exists \varepsilon>0, \forall t\in [0,\varepsilon) \, \lim_{x\to t} h(x,0) = 0\right\} \] 
    contains $f$ but not $g$.
    \item Let $\R$ be the field of real numbers with the exponential map $\exp$ and $X=\mathbb R$. Consider the definable functions $f(x)=x$ and $g(x)=\exp(-1/x^2)$ for $x>0$ and $g(0)=0$. Their zero set agree on $X$ but for each $k\in\mathbb N$ the limit at $0$ of the definable function $h_k=g^k/f$ does not exist. Therefore the prime ideal $\mathfrak p$ given by
     \[ \left\{ h\in \mathcal C^\R(X) \ \big| \ \forall k\ge 1, \exists \varepsilon_k>0, \exists c_k> 1, \forall x\in [0,\varepsilon_k) \ |h(x)|<c_k \cdot x^k \right\} \] 
     contains $g$ but not $f$.
\end{enumerate}
\end{example}
However, for locally closed semialgebraic sets Lojasiewicz's inequality yields the surjectivity of $\iota$. The following result already appears in \cite[Thm.\,8.1]{mT99}. We include a (different) proof for the sake of completeness.

\begin{fact}
Suppose that $\R$ is a polynomially bounded expansion of a real closed field. Assume that $X$ is a locally closed and definable. Then the spectral map $\iota: \Spec_X(\mathcal{R})  \to \mathrm{Spec}(\mathcal{C}^{\mathcal{R}}(X)) $ is surjective and in particular a homeomorphism. 
\end{fact}
\begin{proof}
Let $\mathfrak{p}\in \mathrm{Spec}(\mathcal{C}^{\mathcal{R}}(X))$, and consider the following collection of definable sets
$$\Sigma=\{Z(f) \ | \ f\in \mathfrak{p}\}\cup \{X\setminus Z(f) \ | \ f\notin \mathfrak{p}\}.$$
Let us show that $\Sigma$ is consistent. Assume, towards a contradiction, that 
\[
Z(f_1)\cap \cdots \cap Z(f_n)\cap \big(X\setminus Z(g_1)\big)\cap \cdots\cap  \big(X\setminus Z(g_m)\big)=\emptyset
\] 
for some $f_1,\ldots,f_n\in \mathfrak p$ and some $g_1\ldots,g_m\in \mathcal{C}^{\mathcal{R}}(X)$. Define $f:=f^2_1+\cdots+f^2_n$ and $g:=g_1\cdots g_m$. Thus, $f\in \mathfrak{p}$ and $g\notin \mathfrak{p}$, because $\mathfrak p$ is prime. Note that 
\[
Z(f) = \bigcap_{i=1}^n Z(f_i) \subset \bigcup_{j=1}^m Z(g_j) = Z(g).
\]
Therefore, by \cite[Thm.\,7.3]{mT99} (cf. \cite[Thm.\,2.6.6]{BCR98}, Lojasiewicz's inequality) there is some $k\in \mathbb{N}$ and some $h\in \mathcal{C}^{\mathcal{R}}(X)$ such that $g^k=hf$. In particular, $g\in\mathfrak{p}$, which is a contradiction. 

Hence, take a complete type $p$ extending $\Sigma$. It is obvious that $\mathfrak{p}\subset \iota(p)$, since $\Sigma \subset p$. On the other hand, given $f\in \iota(p)$ we have by definition that $Z(f)\in p$, so $X\setminus Z(f)\notin \Sigma \subset p$. Hence,  $f\in \mathfrak{p}$, yielding that $\iota(p)=\mathfrak p$.
\end{proof}

As shown in \cite[Props.\,2 and 3]{CC83}, the normality of $\mathrm{Spec}(\mathcal C^\R(X))$ implies that the set of maximal ideals $\mathrm{Max}(\mathcal{C}^{\mathcal{R}}(X))$ is quasi-compact and Hausdorff. Furthermore,  there is a continuous retraction map  $r_{\mathrm{Spec}}:\mathrm{Spec}(\mathcal{C}^{\mathcal{R}}(X))\rightarrow \mathrm{Max}(\mathcal{C}^{\mathcal{R}}(X))$.
\begin{theorem}
The map $\iota|_{ \beta_X(\mathcal{R})}: \beta_X(\mathcal{R})\rightarrow \mathrm{Max}(\mathcal{C}^{\mathcal{R}}(X))$ is a homeomorphism and the following diagram is commutative:
\begin{equation}\label{eq:diagram}
\begin{tikzcd}
  \Spec_X(\R) \arrow{r}{\iota} \arrow{d}{r} & \mathrm{Spec}(\mathcal C^\R(X)) \arrow{d}{r_\mathrm{Spec}} \\
           \beta_X(\R) \arrow{r}{\iota|_{ \beta_X(\mathcal{R})}} &   \mathrm{Max}(\mathcal{C}^{\mathcal{R}}(X))
\end{tikzcd}
\end{equation}
\end{theorem}
\begin{proof}
We first show that for every $\mathfrak{m}\in \mathrm{Max}(\mathcal{C}^{\mathcal{R}}(X))$ there exists some $p\in \beta_X(\mathcal{R})$ such that $\iota(p)=\mathfrak{m}$. Indeed, let 
$$\Sigma=\{Z(f) \ | \ f\in \mathfrak{m}\}.$$
Similarly as we argued above $\Sigma$ is consistent. Otherwise, let $f_1,\ldots,f_n\in\mathfrak m$ be such that 
\[
Z(f_1)\cap\ldots\cap Z(f_n) = \emptyset
\]
and set $f=f_1^2+\dots +f_n^2$. It follows that $f\in\mathfrak m$ but $Z(f)=\emptyset$, which is a contradiction. 
Now, let $p\in S_X(\mathcal{R})$ be a complete type extending $\Sigma$. We claim that $\iota(p)=\mathfrak{m}$. Indeed, by definition $\iota(p)$ is a prime ideal that contains $\mathfrak{m}$, so $\iota(p)=\mathfrak m$ by maximality of $\mathfrak m$. Now, since the map $\iota$ is injective and continuous by Proposition \ref{P:iota-cont}, the set $\iota^{-1}(\{\mathfrak{m}\})=\{p\}$ is closed, hence $p\in \beta_X(\mathcal{R})$, as required.

Now we claim that for every $q\in \beta_X(\mathcal{R})$ we have $\iota(q)\in \mathrm{Max}(\mathcal{C}^{\mathcal{R}}(X))$. Indeed, the prime ideal $\iota(q)$ is contained in a maximal ideal $\mathfrak{m}$. Moreover, the previous paragraph and Proposition \ref{P:iota-cont} yield that there is a unique $p\in \beta_X(\mathcal{R})$ such that $\iota(p)=\mathfrak{m}$. We claim that $p\in \mathrm{cl}^{\rm t}(q)$. Otherwise, suppose that there is some definable open set $U$ such that $p\in [U]$ and $q\notin [U]$. Thus, there is some $f\in\mathcal{C}^{\mathcal{R}}(X)$ with $Z(f)=X\setminus U$, so that $Z(f)\notin p$. In particular, by definition $f\notin \iota(p)$, so $f\notin \iota(q)$, because $\iota(q)\subset \mathfrak{m}=\iota(p)$. It follows that $Z(f)=X\setminus U\notin q$, so $U\in q$, which is a contradiction. Therefore, we have proved that $p\in \mathrm{cl}^{\rm t}(q)$ and since $q$ is closed point, we get that $p=q$ and hence $\iota(q)=\iota(p)=\mathfrak{m}$ is maximal, as required.

Altogether, we have proved that
\[\iota|_{ \beta_X(\mathcal{R})}: \beta_X(\mathcal{R})\rightarrow \mathrm{Max}(\mathcal{C}^{\mathcal{R}}(X))\]
is a continuous bijection and so a homeomorphism, because as both spaces are quasi-compact and Hausdorff the map $\iota|_{ \beta_X(\mathcal{R})}$ is closed.

Finally, observe that given a type $q\in \Spec_X(\R)$, the ideal $r_\mathrm{Spec}(\iota(q))$ is the unique maximal ideal which contains $\iota(q)$. Since $r(q)\in \mathrm{cl}^{\rm t}(q)$, it follows that $\iota(q) \subset \iota (r(q))$, so  
\[
\iota(r(q))=r_\mathrm{Spec}(\iota(q))
\]
and diagram \eqref{eq:diagram} is commutative as required.
\end{proof}

We finish this section noting that $\Spec_X(\mathcal{R})$ is locally connected, similarly as it happens with the spectrum of the ring of continuous semialgebraic functions (see \cite[Cor.\,3.10]{BFG22}).

\begin{lemma}\label{spectralLocCon}The topological space $\Spec_X(\mathcal{R})$ is locally connected.	
\end{lemma}
\begin{proof} We first note that if $U$ is a definable open subset of $X$ which is definably connected, then $[U]$ is connected in $\Spec_X(\mathcal{R})$. Otherwise, there are open disjoint non-empty subsets $C_1$ and $C_2$ of $[U]$ such that $[U]=C_1\cup C_2$. Since $C_1$ and $C_2$ are open in $\Spec_X(\mathcal{R})$, we have $C_1=\bigcup_{i\in I}[V_i]$ and $C_2=\bigcup_{j\in J}[W_j]$ for $V_i$ and $W_j$ open definable subsets of $X$. The space $S_X(\mathcal{R})$ is quasi-compact with respect to the Stone topology, so we can assume that $I$ and $J$ are finite. Thus, the sets $V:=\bigcup_{i\in I}V_i$ and $W:=\bigcup_{j\in J}W_j$  are definable, open and closed disjoint non-empty subsets of $U$ such that $U=V\cup W$, which is a contradiction. 
	
Let $p\in\Spec_X(\mathcal{R})$ and let $\widetilde{U}$ be an open neighborhood of $p$ in $\Spec_X(\mathcal{R})$. Then, there is an open definable subset $U$ of $X$ such that $p\in [U]\subset \widetilde{U}$. By \cite[Prop.\,3.2.18]{vdD98} the definable set $U$ has finitely many definably connected components which are open. Therefore, we can assume that $U$ is definably connected. In particular, the set $[U]$ is connected for the spectral topology, as required.
\end{proof}

\subsection{The relation with infinitesimal types}\label{subsec:inftypes}
Fix a definable group $G\subset R^n$ in the o-minimal expansion $\R$ of a real closed field. As usual, by \cite{aP88} we can regard $G$ as a topological group. Moreover, by Robson's embedding theorem \cite[Thm.\,10.1.8]{vdD98} we can assume that the group topology coincides with the topology induced by the ambient space $R^n$. We say that a set $X$ is a {\em $G$-set} if $G$ acts on $X$ and is a {\em $G$-space} if in addition $X$ is a topological space and $G$ acts continuously on $X$ ({\it i.e.} the group action $(g,x)\mapsto g\cdot x$ is a continuous map).

\begin{lemma}
There is an action from $G$ on $\beta_G(\R)$ given by the map $(g,p)\mapsto g\cdot p$. In particular, the set $\beta_G(\R)$ is a $G$-set.
\end{lemma}
\begin{proof}
Given $g\in G$, the map $\pi_g:\Spec_G(\R)\to \Spec_G(\R)$ given by $p\mapsto g\cdot p$ is a homeomorphism. Hence, it maps closed points to closed points. 
\end{proof}

Once we have an action from $G$ on $\beta_G(\R)$, it is natural to ask whether this action is continuous. We will prove that the action is continuous if and only if $\beta_G(\R)$ is canonically isomorphic to the compactification $S^\mu_X(\mathcal{R})$ introduced in \cite{PS17}. 

We recall the definition of $S^\mu_X(\mathcal{R})$ and we refer the reader to \cite[A.1]{PS17} for further details  and its basic properties. 
\begin{definition}
The {\em infinitesimal type} $\mu$ of $G$ is the partial type over $R$ consisting of all $\LL_R$-formulas defining open neighborhoods of the identity. 
\end{definition}

Fix a definable $G$-set $X\subset R^n$ and assume that $G$ acts definably on $X$. Given two $\LL_R$-formulas $\varphi(x)$ and $\psi(x)$ such that $\varphi(R)\subset G$ and $\psi(R)\subset X$, we write $\varphi\cdot \psi$ to denote the $\LL_R$-formula 
\[(\varphi\cdot \psi)(x) = \exists u \exists v (\varphi(u)\wedge \psi(v) \wedge x=uv).\] 
For $p\in S_X(\mathcal{R})$ we define $\mu\cdot p$ as the partial type
\[
(\mu\cdot  p)(x)=\bigcup \left\{ (\theta\cdot \psi)(x) \ | \ \theta\in \mu \text{ and } \psi\in p \right\}.
\]
For $p,q\in S_X(\mathcal{R})$ we say that $p\sim_\mu q$ if $\mu\cdot p$ and $\mu\cdot q$ are equivalent partial types ({\it i.e.}  every formula in $\mu\cdot q$ is implied by $\mu\cdot p$, and vice versa). We denote the equivalence class of $p$ by $[p]_\mu$ and the quotient $S_X(\mathcal{R})/\sim_\mu$ by $S^{\mu}_X(\mathcal{R})$. We consider the quotient topology on it. It is a quasi-compact Hausdorff space, and a basis for its topology is given by open sets of the form
\[
U_\varphi^\mu=\{ [p]_\mu \in S^{\mu}_X(\mathcal{R}) \ | \  \mu \cdot p \vdash \varphi\}
\]
for $\varphi(x)$ an $\LL_R$-formula with $\varphi(R)\subset X$, see \cite[Claim A.3]{PS17}. Moreover, the natural projection $S_X(\mathcal{R})\rightarrow S^{\mu}_X(\mathcal{R})$ is continuous and closed \cite[Claim A.2]{PS17}.

\medskip
Henceforth we will work with the natural definable continuous action of $G$ on itself (even though it is possible to extend our results to other $G$-spaces, see Remark \ref{rmkgeneral}). Our purpose is to compare $\beta_G(\R)$ with $S^{\mu}_G(\mathcal{R})$. As first step, we have the following:
\begin{lemma}\label{L:ContSpecMu}
The natural projection $\pi: \Spec_G(\mathcal{R})\rightarrow S^{\mu}_G(\mathcal{R})$ is continuous and preserves the action of $G$. In particular, the map   $\pi|_{\beta_G(\mathcal{R})}$ is also a continuous surjection. 
\end{lemma}
\begin{proof}
Let $V$ be the preimage of $U_\varphi^\mu$ under $\pi$. Given $q\in V$, since $\mu\cdot q\vdash \varphi$ and $\mu\cdot \mu=\mu$, we can find  by compactness  some $\theta\in \mu$ and $\psi\in q$ such that $(\theta\cdot \theta\cdot \psi)\vdash \varphi$. By assumption, the definable set $(\theta\cdot\psi)(R)$ is open, as so is $\theta(R)$. Hence, $V$ is also open since $q\in [\theta\cdot\psi]\subset V$. We have used that $[\theta\cdot\psi]\subset V$ because $(\theta\cdot \theta\cdot \psi)\vdash \varphi$. 
\end{proof}

One of the main properties of the space $S_G^\mu(\R)$ is that $G$ acts continuoulsy on it \cite[Claim A.5]{PS17}. We now give an equivalent condition for the continuity of the action of $G$ on $\beta_G(\mathcal{R})$.

\begin{prop}\label{lema:contimplieshoemo}The action of $G$ on $\beta_G(\R)$ is continuous if and only if the map $\pi|_{\beta_G(\mathcal{R})}$ is a homeomorphism.
\end{prop}
\begin{proof} Assume that the action of $G$ on $\beta_G(\R)$ is continuous and consider the map $i: G\rightarrow \beta_G(\mathcal{R})$ given by $g\mapsto \text{tp}(g/R)$. Note that $i$ is by Remark \ref{R:DefSep} definably separated (by closed sets). Thus, by \cite[Claim A.12]{PS17} there is a continuous map $i_*:S^\mu_G(\mathcal{R})\rightarrow \beta_G(\mathcal{R})$ that extends $i$. Notice now that $i_* \circ \pi|_{\beta_G(\mathcal{R})}: \beta_G(\R) \to \beta_G(\R)$ is the identity on realized types. Hence, it is the identity, so $\beta_G(\mathcal{R})$ and $S^\mu_G(\mathcal{R})$ are homeomorphic. Reciprocally, if $\pi|_{\beta_G(\mathcal{R})}$ is a homeomorphism, then we deduce since it preserves the action of $G$ that the action of $G$ over ${\beta_G(\mathcal{R})}$ is continuous from the fact the action of $G$ over $S_G^\mu(\R)$ is continuous \cite[Claim A.5]{PS17}.
\end{proof}

Next, we prove that $\beta_G(\R)$ and $S_G^\mu(\R)$ are homeomorphic whenever $G$ is defi\-na\-bly compact. We will need the following topological lemma:
\begin{lemma}\label{lema:sep} Let $Z$ be a definably compact subset of $G$, and let $V$ be an open definable subset of $G$ with $Z\subset V$. Then there is a definable open neighborhood $W$ of the identity such that $WZ\subset V$.
\end{lemma}	
\begin{proof}Let $f:G\times G \rightarrow G$ be the continuous group operation. Note that $f^{-1}(V)$ is an open definable subset of $G\times G$, and $\{e\}\times Z \subset f^{-1}(V)$. Given $\varepsilon>0$ and $z\in G$ we denote $\widetilde{B}_\varepsilon(z):=B_\varepsilon(z)\cap G$. For every $z\in Z$ we can consider
\[
f(z):=\sup\{ \varepsilon>0 : \exists \delta>0, \, \widetilde{B}_\varepsilon(e) \cdot \widetilde{B}_\delta(z) \subset V\}>0.
\]	
Suppose  there is no $\varepsilon>0$ such that for every $z\in Z$ we have $f(z)>\varepsilon$. Then by definable choice there is a definable curve $\alpha:(0,\infty)\rightarrow Z$ such that $f(\alpha(t))<t$ for every $t>0$. By o-minimality there is some $\varepsilon_1>0$ such that $\alpha:(0,\varepsilon_1)\rightarrow Z$ is continuous. As $Z$ is definably compact, we have $z_0:=\lim_{t\to 0}\alpha(t)\in Z$, so there are $\varepsilon_0>0$ and $\delta_0>0$ such that $ \widetilde{B}_{\varepsilon_0}(e)\cdot \widetilde{B}_{\delta_0}(z_0) \subset V$.
Since $\alpha$ is continuous, the open definable set $\alpha^{-1}(\widetilde{B}_{\delta_0}(z_0))$ contains an interval of the form $(0,\varepsilon_2)$. Thus, $\alpha(t)\in \widetilde{B}_{\delta_0}(z_0)$ for every $t<\varepsilon_2$. In particular, for each $t<\varepsilon_2$ there is $\delta_t>0$ with $\widetilde{B}_{\delta_t}(\alpha(t))\subset \widetilde{B}_{\delta_0}(z_0)$, so 
\[
\widetilde{B}_{\varepsilon_0}(e)\cdot \widetilde{B}_{\delta_t}(\alpha(t))\subset  \widetilde{B}_{\varepsilon_0}(e)\cdot \widetilde{B}_{\delta_0}(z_0) \subset V.
\]
It follows that $f(\alpha(t))>\varepsilon_0$ for every $t<\varepsilon_2$, which contradicts the fact that $\varepsilon_0<f(\alpha(t))<t$ for every $t<\min\{\varepsilon_0,\varepsilon_2\}$. 
Consequently, there is an $\varepsilon>0$ such that $W:=\widetilde B_\varepsilon(e)$ satisfies $WZ\subset V$, as required.
\end{proof}

We can now state and prove the main result of this section.

\begin{theorem}\label{teo:samuel}For each type $p\in S_G(\mathcal{R})$ we have that  $r(p) \sim_\mu p$. 	 Furthermore, if $p$ is bounded, then $r(p)$ is the unique type in $\beta_G(\R)$ satisfying this, i.e., \[[p]_\mu \cap \beta_G(\mathcal{R})=\{r(p)\}.\]	
\end{theorem}
\begin{proof}Let $p\in S_G(\mathcal{R})$ and note first that $q\sim_\mu p$ if and only if $q\vdash \mu\cdot p$, see  \cite[Claim 2.7]{PS17}. Thus, the partial type $\mu \cdot p$ determines a closed subset 
\[
[\mu\cdot p ] = \left\{ q\in \Spec_G(\mathcal{R}) \ | \ q \vdash \mu\cdot p \right\}
\]
in $\Spec_G(\mathcal{R})$, because $[\mu\cdot p ]$ is the inverse image of $\{[p]_\mu\}$ under $\pi: \Spec_G(\R)\to S_G^\mu(\R)$, which is continuous by Lemma \ref{L:ContSpecMu}. Hence, we deduce that $\mathrm{cl}^{\rm t}(p)\subset [\mu \cdot p]$, so $r(p)\in [\mu \cdot p]$, which yields that $r(p) \sim_\mu p$.

Assume now that $p$ is bounded, and let us show: $[p]_\mu \cap \beta_G(\mathcal{R})=\{r(p)\}$. We can clearly assume $r(p)=p$ and suppose, to get a contradiction, that there exists some $q\in \beta_G(\mathcal{R})\setminus\{p\}$ such that $q\sim_\mu p$. Since $\beta_G(\mathcal{R})$ is quasi-compact and Hausdorff, there are two disjoint open definable subsets $U_1\subset G$ and $U_2\subset G$ with $p\in [U_1], q\in [U_2]$ and $\overline{U_1}\cap \overline{U_2}=\emptyset$, where $\overline{U_i}$ denotes the topological closure of $U_i$ in $G$ for $i=1,2$. In particular, since $p$ is bounded, we can assume that $U_1$ is also bounded, so that $\overline{U_1}$ is definably compact and contained in the open definable set $G\setminus \overline{U_2}$. By Lemma \ref{lema:sep} there is an open definable neighborhood $W$ of $e$ such that $W\cdot \overline{U_1}\subset G\setminus\overline{U_2}$. On the other hand, note that $W\cdot U_1\in \mu\cdot p$. Since $q\sim_\mu p$, we have $q \vdash \mu \cdot p$ and hence $q\in [W\cdot U_1]$, which is a contradiction, because $U_2\cap  (W\cdot U_1)=\emptyset$.\end{proof}

\begin{cor}\label{C:HomeoBetaMu}Let $G$ be a definably compact group. Then the map $p\mapsto [p]_\mu$ from $\beta_G(\mathcal{R})$ to $S^\mu_G(\mathcal{R})$ is a homeomorphism, so the action of $G$ on  $\beta_G(\mathcal{R})$ is continuous.
\end{cor}

\begin{proof}By Lemma \ref{L:ContSpecMu},  the natural projection  $\beta_G(\mathcal{R})\rightarrow S^{\mu}_G(\mathcal{R})$ is continuous. In addition, it is bijective by Theorem \ref{teo:samuel} and therefore it is a homeomorphism. Hence, the action of $G$ on $\beta_G(\mathcal{R})$ is continuous, which follows from the proof of \cite[Claim A.5]{PS17}.
\end{proof}

In general, if $p$ is not bounded then $[p]_\mu \cap \beta_X(\mathcal{R})$ is not a singleton. We refer to the following example from \cite{PS17}.

\begin{example}
Let $G=\mathbb{R}^2$ and let $\R$ be an elementary extension of the field of real numbers $\mathbb{R}$. Fix some infinite element $a\in R$ and consider the types $p=\tp(a,0/\mathbb R)$ and $q=\tp(a,a^{-1}/\mathbb{R})$. Both types are closed in $\Spec_{G}(\mathbb{R})$ but $p \sim_\mu q$. In particular, the action $\mathbb{R}^2\times \beta_{G}(\mathbb{R})\rightarrow \beta_{G}(\mathbb{R})$ is not continuous by Lemma \ref{lema:contimplieshoemo}.
\end{example}

\begin{remark}\label{rmkgeneral}An inspection of the proof of Lemma \ref{L:ContSpecMu} yields that: If $G$ is a definable o-minimal group, the set $X\subset R^n$ is a definable $G$-space and the action $(g,x)\mapsto g\cdot x$ is definable, continuous and for every $x\in X$ the map $g\mapsto g\cdot x$ is open, then the natural projection map $\beta_X(\mathcal{R})  \rightarrow S^\mu_X(\mathcal{R}) $ is continuous, so it is a homeomorphism. It is enough to adapt the proofs of Theorem \ref{teo:samuel} and Corollary \ref{C:HomeoBetaMu}. 
\end{remark}

\section{Closed, finitely satisfiable and invariant types}\label{sec:invcoh}

Fix an o-minimal expansion $\R$ of a real closed field and let $\bar{\mathcal R}$ be a sufficiently saturated elementary extension. Given a definable subset $X\subset \bar R^n$ definable over $\R$, we can consider the space of closed points $\beta_X(\bar \R)$ within $\Spec_X(\bar \R)$, as well as the usual space of types $S_X(\bar \R)$. In $S_X(\bar \R)$ we have two natural subsets with respect to the structure $\R$. Namely, the set of types $\CohR{X}$ which are finitely satisfiable in $R$ and the set of $R$-invariant types $\InvR{X}$. Recall that a type $p\in S_X(\bar \R)$ is:
\begin{enumerate}[i)]
	\item $R$-\emph{invariant} if for every automorphism $\sigma\in {\rm Aut}_R(\bar{\R})$ of $\bar{\R}$ that fixes $R$ pointwise we have $\sigma(p)=p$, where $\sigma(p):=\{\psi(x,\sigma(b)): \psi(x,b)\in p\}$, and
	\item  \emph{finitely satisfiable} in $R$ if for every formula $\psi(x,b)\in p$ the set of realizations $\psi(R,b):=\psi(\bar{R},b)\cap R$ of $\psi(x,b)$ is non-empty.
\end{enumerate} 
 We briefly recall some well-known facts on invariant and finitely satisfiable types.
 
\begin{remark} \label{rmk:invcoh}  It is very easy to verify that finitely satisfiable types in $R$ are $R$-invariant. Indeed, given a formula $\psi(x,b)$ and some $\sigma\in {\rm Aut}_R(\bar{\R})$, if $p$ is finitely satisfiable in $R$, then the relation $\psi(a,b)\leftrightarrow \psi(a,\sigma(b))$ holds for every $a\in R$. Thus $\psi(x,\sigma(b))\in p$ whenever $\psi(x,b)\in p$. 
	
	The sets $\InvR{X}$ and $\CohR{X}$ are both closed in $S_X(\bar \R)$. For the latter, notice  that 	$$\CohR{X}=\bigcap \left\{ [\neg \psi]  :  \psi \text{ is an $\LL_{\bar R}$-formulas with $\psi(R)=\emptyset$} \right\}.$$
For the former, given any $\sigma\in {\rm Aut}_R(\bar{\R})$, the map $S_X(\bar \R)\rightarrow S_X(\bar \R)\times S_X(\bar \R)$ defined by $p \mapsto (p,\sigma(p))$ is continuous. Thus, the preimage $C_\sigma$ of the dia\-go\-nal of $S_X(\bar \R)\times S_X(\bar \R)$ under this map is a closed subset of $S_X(\bar \R)$, because $S_X(\bar \R)$ is Hausdorff. Hence $\InvR{X}=\bigcap_{\sigma \in{\rm Aut}_R(\bar{\R})}C_\sigma$ is a closed subset of $S_X(\bar \R)$. 	 
\end{remark} 
It will also be  useful to regard  $\CohR{X}$ as the space of types of the externally definable sets. Recall that a set $Z\subset R^n$ is called {\em externally definable} if there exists some $\LL_{\bar R}$-formula $\phi(x)$ such that $Z=\phi(R)$. If $S_X(\R)^\text{ext}$ denotes the set of ultrafilters of externally definable subsets of $X$ with the Stone topology, then the map 
\[\CohR{X}  \rightarrow S_X(\R)^\text{ext} , \ 
p \mapsto  \{\psi(R):\psi(x)\in p\}
\]
is certainly a homeomorphism.

\subsection{Finitely satisfiable types are closed}\label{subsec:cohclosed} We aim to prove that $\CohR{X}$ is contained and closed in $\beta_X(\R)$.  In order to see this, we obtain some results on externally definable sets which are interesting by themselves. We show that every externally definable subset of $R^n$ is given by the trace of an open definable subset of $\bar R^n$. 

The cell decomposition \cite[Ch.\,3]{vdD98} is a fundamental result for the development of o-minimal structures. A cell in the o-minimal structure $\bar\R$ is defined inductively. A subset of $\bar{R}$ is a \emph{cell} if it is a point or an open interval of $\bar{R}$ with endpoints in $\bar{R}\cup \{\pm \infty\}$. More generally, a subset $C\subset \bar{R}^{n+1}$ is a \emph{cell} if there is a cell $D\subset \bar{R}^{n}$ an either there is a continuous  $\bar{R}$-definable function $f:D\rightarrow \bar{R}$ such that $C$ is the graph of $f$, or there are two continuous $\bar{R}$-definable functions $f_1,f_2:D\rightarrow \bar{R}$ with $f_1(x)<f_2(x)$ for all $x\in D$ such that $C=\{(x,y):x\in D \text{ and } 
 f_1(x)<y<f_2(x)\}$.
\begin{lemma}
	Let $ C\subset \bar R^n$ be a cell. Then there exists some open cell $U\subset  \bar R^n$ such that $C\subset U$ and $C\cap R^n=U\cap R^n$. Moreover, there is also a definable continuous retraction $r:U\to C$.
\end{lemma}
\begin{proof} 
	We proceed by induction on $n$. For $n=1$, suppose first that $C=\{a\}$. In this case set $U=(a-\delta,a+\delta)$ for some $\delta>0$ infinitesimal with respect to a model containing $R$ and $a$. It is clear that $U$ and the definable constant map $r:U\to C$ defined by $u\mapsto a$ satisfy the requirements.  On the other hand, if $C$ is an open interval, then it suffices to set $U=C$ and $r$ as the identity map.  
	
	Assume that the statement holds for $n$ and let $C$ be a cell of $\bar R^{n+1}$. Suppose first that $C$ is the graph $\Gamma(f)$ of a definable continuous function $f:D\to \bar R$ with $D\subset \bar R^n$ a cell. By induction, there is a definable open cell $U$ of $\bar R^n$ containing $D$ such that $D\cap R^n=U\cap R^n$ and a definable continuous retraction $r:U\to D$. Set $g=f\circ r$, which is a definable continuous function such that $g|_{D} = f$. Set $V=(g-\delta,g+\delta)_U$ for some $\delta>0$  infinitesimal with respect to $R$, that is, 
	\[
	V=\{ (\bar x,y)\in  U\times \bar R \, | \, y\in (g(\bar x)-\delta,g(\bar x)+\delta) \}.
	\]
	It is clear that $V\subset \bar R^{n+1}$ is an open cell such that $C\subset V$. Moreover, since $U\cap R^n =D\cap R^n$ and $g|_{D} = f$, it holds $V\cap R^{n+1}\subset C\cap R^{n+1}$. To obtain a definable continuous retraction $\tilde r:V\to C$ it is enough to consider the map $(\bar x,y)\mapsto (r(\bar x), g(\bar x))$.
	
	Finally, suppose that $C=(f_1,f_2)_D$ for some cell $D\subset \bar R^n$ and some definable functions $f_1,f_2:D\to \bar R$ with $f_1<f_2$. By induction hypothesis, there exists a definable open cell $U$ of $\bar R^n$ containing $D$ such that $D\cap R^n=U\cap R^n$ and a definable continuous retraction $r:U\to D$. For $i=1,2$ set $g_i=f_i\circ r:U\to \bar R$ and note that $g_i|_{D} = f_i$ for $i=1,2$. Now, take $V=(g_1,g_2)_U$, that is, 
	\[
	V=\{ (\bar x,y)\in \bar U\times \bar R \, | \, y\in (g_1(\bar x),g_2(\bar x)) \}.
	\]
	It holds  $C\subset V$ and that $V\cap R^{n+1}\subset C\cap R^{n+1}$. Now, to obtain a definable retraction $\tilde r:V\to C$ it is enough to consider the map $(\bar x,y)\mapsto (r(\bar x), y)$, as required.
\end{proof}

\begin{prop}\label{P:ExtDefOpen}
	Let $Z \subset R^n $ be an externally definable set given by a set $W$ of $\bar R^n$, i.e. $W\cap R^n=Z$. Then there is a definable open set $U$ of $\bar R^n$ such that $W\subset U$ and $Z=U\cap R^n$. 
\end{prop}
\begin{proof}
   	Let $\mathcal C=\{C_1,\ldots,C_r\}$ be a cell decomposition of $\bar R^n$ compatible with $W$. For each $i$ let $U_i$ be the open cell given by the previous lemma containing $C_i$ and note that
	\[
	Z = W\cap R^n = \bigsqcup_{i=1}^r C_i\cap R^n = \bigsqcup_{i=1}^r U_i\cap R^n.
	\]
	Thus, the definable open set $U:=U_1\cup \ldots \cup U_r$ yields the result.
\end{proof}
As a consequence, we obtain:
\begin{cor}\label{C:CohBeta0}
The collection $$\left\{ [\phi]\cap \CohR{X} \ | \ \phi\in \LL_{\bar R}, \ \phi(\bar R)\subset X \text{ is open and definable}  \right\}$$ is a basis of the space $\CohR{X}$ with respect to the Stone topology. In particular, the subspace topologies on $\CohR{X}$ induced from both the Stone and the spectral topologies coincide.
\end{cor}

We also deduce the following:

\begin{cor}\label{C:CohBeta}
 The set $\CohR{X}$ is closed in $\Spec_X(\bar \R)$ and contained in $\beta_X(\bar \R)$. Namely,
 \begin{enumerate}[(i)]
    \item  $\mathrm{cl}^{\rm t}(\CohR{X}) = \CohR{X}$ and
    \item  for every $p\in \CohR{X}$ we have that $r(p)=p$.
 \end{enumerate}
\end{cor}
\begin{proof}
(i) We prove that \[
\CohR{X}= \bigcap\left\{ [\neg \psi] \ | \ \psi\in\LL_{\bar R}, \psi(\bar R) \subset X \text{ open and } \psi(R)=\emptyset \right\}.
\]
It is clear that the inclusion $\subset$ holds. For the other inclusion, suppose $p\notin \CohR{X}$ and let $\varphi(x)$ be an $\LL_{\bar R}$-formula witnessing this. By Proposition \ref{P:ExtDefOpen} there is some $\LL_{\bar R}$-formula $\psi(x)$ such that $\psi(\bar R)$ is open with $\varphi(\bar R)\subset \psi(\bar R)$ and $\varphi(R)=\psi(R)$. Hence, $p\in[\psi]$, which yields the desired equality.

(ii) By (i) the set $\CohR{X}$ is closed in $\Spec_{X}(\bar \R)$. Thus, a point of $\CohR{X}$ is closed in $S_{X}(\bar \R)_\R^{\rm fs}$ if and only if it is closed in $\Spec_{X}(\bar \R)$. On the other hand, the previous corollary asserts that the subspace topologies on $\CohR{X}$ induced from both the Stone and the spectral topology coincide. Consequently, a point of $\CohR{X}$ is always closed in $\CohR{X}$, as required.
\end{proof}
\subsection{Invariant types}\label{subsec:inv} Once we have seen that finitely satisfiable types are closed within the spectral topology, it is natural to ask which is the relation between them and closed invariant types, as well as the relation between invariant and closed types. We first prove:

\begin{lemma}The set $\InvR{X}$ is a closed subspace of $S_X^{\rm t}(\bar\R)$. In particular, the set $\InvRt{X}$ of $R$-invariant types with the induced spectral topology from $S^{\rm t}_X(\bar \R)$ is a normal spectral space.\end{lemma}

\begin{proof}It is enough to prove that for $p\in \InvR{X}$ its closure ${\rm cl}^{\rm t}(p)$ in $S^{\rm t}_X(\bar \R)$ is contained in $\InvR{X}$. Indeed, if $q\in {\rm cl}^{\rm t}(\InvR{X})$, then for each open $\bar R$-definable subset $U\subset X$ with $q\in [U]$ we have $[U]\cap \InvR{X}\neq \emptyset$. Since $\InvR{X}$ is a closed subset of $S_X(\bar \R)$ with the Stone topology,  by quasi-compactness there is a type $p\in \InvR{X}$ that is in all the open neighbourhoods in $S^{\rm t}_X(\bar \R)$ of $q$ and so  $q\in {\rm cl}^{\rm t}(p)\subset \InvR{X}$.
	
Now, let $p\in \InvR{X}$. Consider $\sigma\in {\rm Aut}_{R}(\bar\R)$ and note that $S^{\rm t}_X(\bar \R)\rightarrow S^{\rm t}_X(\bar \R)$ given by $p\mapsto \sigma(p)$ is a homeomorphism. Thus $\sigma({\rm cl}^{\rm t}(p))={\rm cl}^{\rm t}(\sigma(p))={\rm cl}^{\rm t}(p)$. It is also clear that $\sigma(q_1)\in \mathrm{cl}^{\rm t}(\sigma(q_2))$ whenever $q_1\in \mathrm{cl}^{\rm t}(q_2)$. Therefore $\sigma$ respects the  specialization order, so by Lemmas \ref{lemanormal} and \ref{L:DimBeta} the set ${\rm cl}^{\rm t}(p)$ is a finite set totally ordered under specialization. We deduce that $\sigma$ fixes 
${\rm cl}^{\rm t}(p)$ pointwise, as required. 

The fact that $\InvRt{X}$ is a spectral space follows from  \cite[Thm.\,2.1.3]{DST19}, as $\InvR{X}$ is closed in $S_X(\bar\R)$ with respect to the Stone topology. Since $\InvRt{X}$ is closed in $S^{\rm t}_X(\bar\R)$ and the latter is normal, we deduce that $\InvRt{X}$ is normal.
\end{proof}

\begin{remark}\label{R:InvRetraction}
The statement above yields that $\InvR{X}$ is closed under images of the map $r:\Spec_X(\bar \R)\to\beta_X(\bar R)$ and that the set of closed points of $\InvRt{X}$ agrees with $\InvRt{X}\cap \beta_X(\bar \R)$. As a consequence, the natural retraction from $\InvRt{X}$ onto its closed points coincides with $r|_{\InvRt{X}}$.\end{remark}

In general, to be a closed point of  $\InvRt{X}$ depends on the set $X$, see Example \ref{E:InvNotClosed} below. However, when $X$ is closed, we prove the following:

\begin{theorem}\label{T:InvBddCoh}
Assume that $X\subset \bar R^n$ is closed. A type is $R$-invariant, bounded and closed if and only if it is finitely satisfiable in $R$, that is: 
\[
\InvR{X} \cap \beta_X(\bar \R) \cap S_X(\bar \R)^{\rm bdd}= \CohR{X},
\]
where $S_X(\bar \R)^{\rm bdd}$ is the set of bounded types. 
\end{theorem}
\begin{proof}
By Corollary \ref{C:CohBeta} we have $\CohR{X}\subset \beta_X(\bar\R)$. Hence to prove that the inclusion $\supset$ of the statement holds, it suffices to see that any $p\in \CohR{X}$ is bounded. Choose some $r\in \bar R$ with $r>\|m\|$ for every $m\in X\cap R^n$. Thus the definable subset
\[
V=\{x\in X \ | \ \|x\|>r \}
\]
is not realized in $\R$, hence $p\in [V^c]$, which implies that $p$ is bounded. 

To prove the other inclusion, fix some $p\in \InvR{X}\cap \beta_X(\bar \R)\cap S_X(\bar \R)^{\rm bdd}$. Let $\varphi(x)$ be an arbitrary $\LL_{\bar R}$-formula such that $p\in [\varphi]$, and let us show that $\varphi(R)$ is non-empty. Clearly, we may take $\varphi(\bar R)$ to be bounded. As $\mathrm{cl}^{\rm t}(p)=\{p\}$, there exist some $\LL_{\bar R}$-formulas $\psi_i(x)$ such that each $\psi_i(\bar R)\subset X$ is closed in $X$ and 
\[
\{ p \} = \bigcap_{i\in I} [\psi_i]\]
in $\Spec_X(\bar\R)$. Thus,  we have that  $\{\psi_i(x) \ | \ i\in I \} \cup \{ \neg\varphi(x)\}$ is inconsistent with $\mathrm{Th}(\bar \R)$, so there is some finite $I_0\subset I$ such that, for
\[
\psi(x):=\bigwedge_{i\in I_0} \psi_i(x), 
\]
we have  $\mathrm{Th}(\bar \R)\cup\{\psi(x)\}\vdash \varphi(x)$ .
Note that  $p\in [\psi]$ and $\psi(\bar R)\subset \varphi(\bar R)$ is closed in $\bar R^n$ as so is $X$. Thus, as $p$ is $R$-invariant, by \cite[Thm.\,6.5]{PP07} (see also \cite[Thm.\,3.5]{aD04} or \cite{sS08}) we deduce that $\psi(R)\neq\emptyset$, so $\varphi(R)\neq\emptyset$ and consequently $p\in \CohR{X}$.
\end{proof}
The following easy example shows that in general the statement above fails for unbounded types, as well as for non-closed sets:
\begin{example}\label{E:InvNotClosed}
Let $\R$ be the real closed field and let $\bar \R$ denote a sufficiently saturated and homogeneous elementary extension. 
\begin{enumerate}[i)]
    \item Set $X=\bar R$ and let $p_1\in S_X(\bar \R)$ be the type at infinity, that is, the type $p_1(x)$ is determined by all formulas $a<x$ for $a\in \bar R$. Then $p_1$ is unbounded, $R$-invariant, closed in $\Spec_X(\bar \R)$, but it is not finitely satisfiable in $R$.
    \item Set $Y=(0,\infty)$ and let $p_2\in S_Y(\bar \R)$ be the infinitesimal type, that is, the type $p_2(x)$ is determined by formulas of the form $0<x\wedge x<a$ for $a\in \bar R$. Then $p_2$ is bounded, $R$-invariant, closed in $\Spec_Y(\bar \R)$, but it is not finitely satisfiable in $R$.
\end{enumerate}
\end{example}

When the set $X\subset R^n$ is bounded and closed, Remark \ref{R:InvRetraction} and Theorem \ref{T:InvBddCoh} yield that the retraction $r|_{\InvR{X}}$ maps invariant types to  finitely satisfiable ones. Summarizing, we get the following statement:
\begin{cor}\label{C:DefCompInvBetaCoh}
Assume $X\subset R^n$ is definably compact ({\it i.e.} definable, closed and bounded).  Then \[\InvR{X} \cap \beta_X(\bar \R) = \CohR{X},\]
and therefore there is a continuous retraction $r|_{\InvRt{X}}:\InvRt{X}\rightarrow \CohR{X}$.\end{cor}

We stress out that by Corollary \ref{C:CohBeta}  the spectral and the Stone topologies coincide on $\CohR{X}$ and hence there is no need to distinguish them. 

To conclude the section we prove that the retraction $r|_{\InvRt{X}}$ is essentially unique when $X\subset R^n$ is definably compact.

\begin{prop}
Assume $X\subset R^n$ is definably compact. There is a unique continuous retraction from  $\InvRt{X}$ onto $\CohR{X}$. In particular, it is $r|_{\InvR{X}}$.
\end{prop}
\begin{proof}
Let $f$ denote a continuous retraction from  $\InvRt{X}$ onto $\CohR{X}$. By continuity $f$ preserves specializations, that is, if $q\in\mathrm{cl}^{\rm t}(p)$ then $f(q)\in\mathrm{cl}^{\rm t}(f(p))$ for any $p,q\in \InvRt{X}$. Note that $r(p)\in \mathrm{cl}^{\rm t}(p)$, so 
\[f(r(p))\in \mathrm{cl}^{\rm t}(f(p))=\{f(p)\}\] since $f(p)\in  \beta_X(\bar \R)$ by Corollary \ref{C:CohBeta}. As $f$ is a retraction and we have by Theorem \ref{T:InvBddCoh} that $r(p)\in \CohR{X}$, we deduce  $f(r(p))=r(p)$. Altogether we conclude $r(p)=f(p)$, as desired.
\end{proof}

\section{An honest topology}\label{sec: honesty}

The results of the previous section yield the existence of a continuous retraction $$r|_{\InvRt{X}}:\InvRt{R} \to \CohR{X},$$ when $X\subset \bar R^n$ is a definably compact set definable in an o-minimal saturated extension $\bar \R$ of a real closed field $\R$. This retraction comes from considering the spectral topology on the space of types and it captures in a way the fact of being infinitesimally close. It does not seem feasible to transfer this construction for arbitrary NIP structures  $\M\preceq \bar\M$ due to the absence of a natural topology on $\bar \M$.  Nonetheless, Simon \cite{pS15} is able to construct for NIP structures a canonical conti\-nuous retraction $F_M: \Inv{X}\to \Coh{X}$. His construction lacks {\it a priori} of a topological interpretation. The purpose of this final section is to find a topological interpretation of $F_M$ within the context of spectral spaces.  

Let $T$ be a complete $\LL$-theory with NIP. Henceforth we fix a model $\mathcal{M}$ of $T$ and a saturated elementary extension $\bar{\mathcal{M}}$ of $\mathcal{M}$. For a subset $X\subset M^n$ definable in $\mathcal{M}$, the retraction
$$F_M: \Inv{X}\to \Coh{X}$$
is continuous with respect to the Stone topology and satisfies $F_M(p)|_M = p|_M$. We refer to \cite[Sec.\,3]{pS15} for further details (cf. \cite{CPS14}), but we briefly recall some details of Simon's construction.

\begin{remark}\label{FMdescription}
 Let $\chi(x)$ be an $\LL_{\bar{M}}$-formula and fix a saturated model $M'$ of $M$ inside $\bar{\M}$ that contains the parameters of $\chi$. Consider a new unary predicate $P$ and set $\LL_P=\LL\cup\{P\}$ and fix the $\LL_P$-structure $(\M',\M)$ with $P(M')=M$. Choose inside $\bar \M$ a sufficiently saturated elementary extension $(\mathcal N',\mathcal N)$ of $(\M',\M)$. In \cite[Sec.\,3.1]{pS15} it is proven that for every $p\in \Inv{X}$ both $p|_N(x)\cup \{P(x),\chi(x)\}$ and $p|_N(x)\cup  \{P(x),\neg\chi(x)\}$ cannot be consistent. In particular, either 
\begin{equation}\label{piefor} p|_N(x)\cup  \{P(x)\}\vdash \chi(x) \qquad \text{or}\qquad p|_N(x)\cup \{P(x)\}\vdash\neg\chi(x).  \tag{$\ast$}
\end{equation}
Therefore there is a unique type $F_M(p)\in S_X(\M')$ such that $p|_N\cup \{P(x)\}\cup F_M(p)$ is consistent. Note that $F_M(p)$ is  finitely satisfiable in $M$ as it is consistent with $P$.

\end{remark}

The following example shows that the retraction $F_M$ is certainly different from $r|_{\InvRt{X}}$.

\begin{example}\label{contraejemplo}
Let $\R$ be the real field and consider some sufficiently saturated model $\bar \R$ of its theory. Let $X=[0,1]$ and let $p(x)$ be the complete type determined by the formulas $0<x$ and $x<a$ for every $a\in \bar R_{>0}$. This is an $R$-invariant type and satisfies that $r(p)=\tp(0/\bar R)$. On the other hand, the formula $x=0$ does not belong to $p|_M$, so $r(p)|_M \neq p|_M= F_M(p)|_M$. Hence, both retractions $r$ and $F_M$ are distinct.  
\end{example}

Our goal is to equip $\Inv{X}$ with a topology of normal spectral space and to show that the natural retraction to the closed points is Simon's retraction $F_M$.

\begin{definition}
An $\LL_{\bar M}$-formula $\theta(x)$ is {\em honest} over $M$ if for any $\LL_M$-formula $\varphi(x)$ we have $\theta(\bar M)\subset \varphi(\bar M)$ whenever $\theta(M)\subset \varphi(M)$.  
\end{definition}
By definition if $\theta(M)$ is empty, then $\theta(\bar{M})$ must be also empty. 
Each $\LL_{M}$-formula, and  each $\LL_{\bar M}$-formula $\theta(x)$ with $\theta(M)=M^{|x|}$, is honest. Furthermore, the following result of Chernikov and Simon yields the existence of honest formulas (the third clause appears in \cite[Prop.\,3.11]{CPS14}):

\begin{fact}\cite[Prop.\,1.7]{CS13}\label{F:Honest_exist}
For any $\LL_{\bar M}$-formula $\psi(x)$ there exists an $\LL_{\bar M}$-formula $\theta(x)$ honest over $M$ such that $\theta(M)=\psi(M)$. Moreover, there is an honest formula $\theta'(x)\in \LL_{\bar M}$ such that
\begin{enumerate}[i)]
	\item $\theta(M)=\psi(M)=\neg\theta'(M)$,
	\item $\theta(\bar{M})\subset \neg\theta'(\bar{M})$ and 
	\item no $M$-invariant type contains the formula $\neg\theta'\wedge \neg\theta$.
\end{enumerate}
\end{fact}

Since it is clear from the context, we shall omit the reference ``over $M$'' when talking about honestity. As a finite disjunction of honest formulas is again honest, we can equip the set $S_n(\bar \M)$ with a new topology. 

\begin{definition}
The {\em honest topology over $M$} on $S_X(\bar \M)$ has as a basis of closed sets the sets of the form $[\theta]$ where $\theta(x)$ is honest over $M$.  We denote the space of types with the honest topology by $\HS_X(\bar \M)$. 
\end{definition}

\begin{prop}
The topological space $\HS_X(\bar \M)$ is spectral.
\end{prop}
\begin{proof}To see that it is $\mathrm{T}_0$, let $p$ and $q$ be different types in $S_X(\bar{\M})$. We show that there is an honest $\LL_{\bar{ M}}$-formula $\theta$ so that $p\in[\theta]$ and $q\notin [\theta]$, or $p\notin[\theta]$ and $q\in [\theta]$. Since $p\neq q$, there is an $\LL_{\bar{ M}}$-formula $\psi$ such that $\psi\in p$ and $\neg \psi \in q$. We first consider the cases when $p$ is not finitely satisfiable in $M$. Then we can find an $\LL_{\bar{ M}}$-formula $\psi_0\in p$ such that $\psi_0(M)=\emptyset$ and $\psi_0(\bar M)\subset \psi(\bar M)$. Thus $\neg\psi_0$ is honest and clearly $q\in [\neg \psi_0]$, which yields the result. Likewise, we obtain the result whenever $q$ is not finitely satisfiable. Therefore we may assume that both $p$ and $q$ are finitely satisfiable in $M$, in which case, it is enough to take honest definitions $\theta_1$ and $\theta_2$ of $\psi$ and $\neg\psi$ respectively. It follows that $p\in [\theta_1]\setminus [\theta_2]$ and $q\in[\theta_2]\setminus [\theta_1]$, as desired.

The map $S_X(\bar{\M}) \to \HS_X(\bar \M)$ given by $p\mapsto p$ is continuous and thus $\HS_X(\bar \M)$ is quasi-compact. Hence we get (S1). Moreover, the continuity of the map above yields that every basic open set is quasi-compact. Hence, the set of all quasi-compact open sets is the set of finite unions of basic open sets. Thus (S2) and (S3) follow. Finally, to show that $\HS_X(\bar \M)$ is sober, let $C$ be a nonempty closed and irreducible subset of $\HS_X(\bar \M)$. It is straightforward to check that the set
\[
\Sigma(x)=\{ \theta(x)\in \LL_{\bar M} \ | \ C\subset [\theta],\  [\theta] \text{ honest} \} \cup \{  \neg\theta(x)\in \LL_{\bar M} \ | \ C\not\subset [\theta],\ [\theta] \text{ honest}\}
\]
is consistent. Therefore we can complete it into a type $p\in S_X(\bar\M)$. Clearly $p\in C$ because $p$ belongs to all the basic closed sets containing $C$, so $\text{cl}^{\text{h}}(p)\subset C$. Suppose there is $q\in C$ such that $q\notin \text{cl}^{\text{h}}(p)$. Then there is an honest formula $\theta_0$ such that $q\in [\neg \theta_0]$ and $p\in [\theta_0]$. By definition of $\Sigma(x)$ and since $p\in [\theta_0]$, we deduce $C\subset [\theta_0]$, so $q\in [\theta_0]$, which is a contradiction. It follows that $\text{cl}^{\text{h}}(p)= C$, as required. \end{proof}

\begin{remark}
The induced honest topology and the constructible topology on $\Coh{X}$ coincide  and hence there is no need to distinguish them. In fact, one can prove that every basic closed set in $\Coh{X}$ is open. Indeed, a basic closed subset is of the form $\Coh{X}\cap [\theta]$, for some honest formula $\theta$. We may assume  $\Coh{X}\not\subset [\theta]$. Thus, $\neg\theta(M)\neq\emptyset$, so we can find an honest definition $\psi$ for $\neg\theta$. Since $\psi(M)=\neg\theta(M)$, we deduce $\Coh{X}\cap [\neg\theta] = \Coh{X}\cap [\psi]$ is also closed. 
\end{remark}

\begin{lemma}\label{cohclosed} Every  finitely satisfiable type in $M$ is a closed point of $\HS_X(\bar \M)$.
\end{lemma}
\begin{proof} Let $p$ be finitely satisfiable in $M$. We show that it is closed. Pick $q\in \text{cl}^\text{h}(p)$ and suppose that $q$ is not finitely satisfiable in $M$. Thus, there is some formula $\psi\in q$ such that $\psi(M)=\emptyset$. Since $\neg \psi$ is  honest, the type $q$ belongs to the open set $[\psi]$ and $p\notin [\psi]$, we get a contradiction. Hence, the type $q$ is finitely satisfiable in $M$ as well. This yields that $p=q$, since $\Coh{X}$ with the inherited honest topology is a Hausdorff space. 
\end{proof}

\begin{example}In general, it is not true that the closed points coincide with the  finitely satisfiable types in $M$. For example, if $M=\mathbb{R}_{\text{alg}}$ and $\bar{\M}$ is a saturated elementary extension containing $\mathbb{R}$, the type $\text{tp}(\pi/{\bar{M}})$ is not finitely satisfiable in $M$. However, both intervals $[0,\pi]$ and $[\pi,4]$ of $\bar{M}$ are defined by honest (over $M$) $\LL_{\bar \M}$-formulas, and $[0,\pi]\cap [\pi,4]=\{\pi\}$, so that $\text{tp}(\pi/{\bar{M}})$ is closed in $\HS_X(\bar \M)$. One can check that the honest topology is normal in this example. 
\end{example}

In the light of the example above, we ask:
\begin{quest}
Is in general $\HS_X(\bar \M)$ a normal topological space?
\end{quest}

Henceforth, the space $\Inv{X}$ with the subspace honest topology is denoted by $\Invh{X}$. Next, we prove that it is a spectral topological space. In addition, we will also see in Theorem \ref{thm:normal} that it is normal. 
\begin{prop}The topological space $\Invh{X}$ is spectral.
\end{prop}
\begin{proof} This follows from \cite[Thm.\,2.1.3]{DST19}. 
Nevertheless, we give a direct proof for completeness. Properties (S1)-(S3) follow from the fact that $\HS_X(\bar \M)$ is a spectral space and $\Inv{X}$ is closed in $S_X(\bar \M)$ with the Stone topology. Let us show (S4). Pick a closed and irreducible set $C_0$ of $\Invh{X}$. Consider its closure $C:=\text{cl}^{\text{h}}(C_0)$ in $\HS_X(\bar \M)$. We clearly have that the closed and irreducible subset $C$ of $\HS_X(\bar \M)$ satisfies $C\cap \Inv{X}=C_0$. Hence, since $\HS_X(\bar \M)$ is a spectral space, there is a unique $p\in \HS_X(\bar \M)$ such that $\text{cl}^\text{h}(p)=C$. Let us see that $p\in \Inv{X}$. For each  $\sigma\in {\rm Aut}_M(\bar{\M})$ the induced map $\HS_X(\bar \M)\rightarrow \HS_X(\bar \M)$ given by $p\mapsto \sigma(p)$ is a homeomorphism. Indeed, it is enough to note that if $\theta(x,d)$ is honest, then so is $\theta(x,\sigma(d))$. Hence, it follows that $$\textrm{cl}^{\rm h}(\sigma(p))=\sigma(\textrm{cl}^{\rm h}(p))=\sigma(C)=\sigma(\textrm{cl}^{\rm h}(C_0))=\textrm{cl}^{\rm h}(\sigma(C_0))=\textrm{cl}^{\rm h}(C_0)=C$$
	and by uniqueness of $p$ we deduce that $\sigma(p)=p$. Thus $p\in C\cap \Inv{X}=C_0$, so $C_0$ is the closure of $p$ in $\Invh{X}$, as required.
\end{proof}	

We next prove normality of $\Invh{X}$. To that end we need to characterize Simon's $F_M$ retraction.

\begin{prop}\label{honestF_M} Let $p\in \Inv{X}$. An $\LL_{\bar{ M}}$-formula $\psi(x)$ belongs to $F_M(p)$ if and only if there is an honest $\LL_{\bar{ M}}$-formula $\theta(x)\in p$ such that $\theta(M)=\psi(M)$.
\end{prop}
\begin{proof} \noindent ($\Rightarrow$) Let $\psi(x)$ be an $\LL_{\bar{ M}}$-formula and suppose $\psi(x)\in F_M(p)$. Let $\M'\preceq \bar \M$ be an $|M|^+$-saturated elementary extension of $\M$ containing the parameters of $\psi(x)$. Now, working in the language $\LL_P=\LL\cup\{P\}$, consider a sufficiently saturated elementary extension $(\mathcal N',\mathcal N)$ of $(\mathcal M',\mathcal M)$. Set
	$$S:=\left\{q\in \Inv{X}: q|_N(x)\cup \{P(x)\}\cup \{\psi(x)\} \text{ is consistent}\right\}.$$
	By \eqref{piefor} in Remark \ref{FMdescription} for every $q\in S$ there is some $\LL_N$-formula $\theta_q$ such that $\theta_q(x)\wedge P(x) \vdash \psi(x)$. Since $S$ is a closed subset of $\Inv{X}$ and $S\subset \bigcup_{q\in S} [\theta_q]$, by quasi-compactness there is an $\LL_{N}$-formula $\theta(x)$ such that $\theta(x)\wedge P(x)\vdash \psi(x)$ and $S\subset [\theta]$. In particular, there is no type $q\in \Inv{X}$ for which the set $q|_N(x)\cup \{P(x)\}\cup \{\psi(x)\wedge \neg\theta(x)\}$ is consistent.

	By definition of $F_M(p)$ we have $p\in S$, so $\theta\in p$. Therefore, it is enough to show that $\theta$ is an honest and $\theta(M)=\psi(M)$. For the latter, if $a\in \psi(M)$ then $\tp(a/\bar{M})\in S$ and so $a\in \theta(M)$. On the other hand, if $a\in \theta(M)$, then, since $\theta(x)\wedge P(x)\vdash \psi(x)$, we deduce $a\in \psi(M)$. Hence $\theta(M)=\psi(M)$. Finally, to prove that $\theta(x)$ is honest, let $\varphi(x)$ be an $\LL_{M}$-formula with $\psi(M)=\theta(M)\subset \varphi(M)$. Then $(\mathcal M',\mathcal M)$ and so $(\mathcal N',\mathcal N)$ satisfy 
	\[\forall x (\psi(x)\wedge P(x)\rightarrow \varphi(x)).\] As $\theta(x)\wedge P(x)\vdash \psi(x)$, we deduce that $(N',N)$ satisfies $\forall x (\theta(x)\wedge P(x)\rightarrow \varphi(x))$ and, since $\theta$ has parameters in $N$, we deduce  $\theta(\bar{M})\subset \varphi(\bar{M})$.	
	
\noindent ($\Leftarrow$) Fix $\psi(x)$ and let $\theta(x,b)\in p$ be an honest $\LL_{\bar{ M}}$-formula such that $\theta(M,b)=\psi(M)$. As before, consider an $|M|^+$-saturated elementary extension  $\M'\preceq \bar \M$ of $\M$ containing the parameters of $\psi(x)$ and $\theta(x,b)$. Also, consider a sufficiently saturated elementary extension $(\mathcal N',\mathcal N)$ of $(\mathcal M',\mathcal M)$. 

Suppose that $\psi(x)\not\in F_M(p)$. As $\psi(M)=\theta(M,b)$, we deduce $\neg \theta(x,b)\in F_M(p)$. By the implication ($\Rightarrow$),  there is some honest $\theta_0(x,c)\in p|_N$  such that $\theta_0(M,c)=\neg \theta(M,b)$ and $\theta_0(x,c)\wedge P(x)\vdash \neg\theta(x,b)$. Since $p$ is $M$-invariant and $p\in [\theta(x,b)\wedge \theta_0(x,c)]$, the formula $\theta(x,b)\wedge \theta_0(x,c)$ is consistent. Set 
\[\varphi(y,z):= \exists x (\theta(x,y)\wedge \theta_0(x,z)). 
\]	
The pair $(N',N)$ satisfies the $\LL_{M'}$-sentence
$$\exists z (P(z)\wedge \varphi(b,z) \wedge \forall x \big(\theta_0(x,z)\wedge P(x)\rightarrow \neg\theta(x,b)\big),$$	
and so does $(M',M)$. Therefore there exists $m\in M$ such that 
$\varphi(b,m)$ holds and $\theta_0(M,m)\subset \neg\theta(M,b)$. In particular, we obtain that $\theta(\bar M,b)\cap \theta_0(\bar M,m)\neq\emptyset$ by the definition of $\varphi(y,z)$  and also that $\theta(M,b)\subset \neg\theta_0(M,m)$. However, as $\theta(x,b)$ is honest and $\theta_0(x,m)$ is an $\LL_{M}$-formula, the latter yields $\theta(\bar{M},b)\subset \neg\theta_0(\bar{M},m)$, which is a contradiction.	
\end{proof}	
As an immediate consequence, we obtain the following:
\[
F_M(p) \in \bigcap \left\{  [\theta] : {\theta\in p\text{ is honest}}  \right\} = \mathrm{cl}^{\rm h}(p).
\]
 We will use the following notation: given an $\LL_{\bar{M}}$-formula $\psi(x)$, we write
$$[\psi]^{\rm inv}:=\{p\in \Inv{X}: \psi\in p\}.$$ 
\begin{cor}\label{interhonest} If $\theta_1,\ldots,\theta_\ell$ are honest over $M$ formulas and $[\theta_1]^{{\rm inv}}\cap \cdots \cap [\theta_\ell]^{{\rm inv}}$ is non-empty, then 
$$(\theta_1\wedge \cdots \wedge \theta_\ell)(M)\neq \emptyset.$$	
\end{cor}
\begin{proof} By assumption, there is an $M$-invariant type $p\in [\theta_i]^{{\rm inv}}$ for $i=1,\ldots,\ell$, so by Proposition \ref{honestF_M} we have $\theta_1,\ldots,\theta_\ell \in F_M(p)$. Since $F_M(p)$ is finitely satisfiable in $M$, we deduce  $(\theta_1\wedge \cdots \wedge \theta_\ell)(M)\neq \emptyset$, as required.	
\end{proof}
Recall that if $X$ is a normal spectral space, then the set of the closed points $\text{Max}(X)$ is a quasi-compact Hausdorff topological space. Moreover, for every $x\in X$ there is a unique closed point $r(x)\in \text{cl}(x)$ and the map $r:X\rightarrow \text{Max}(X)$ is a continuous retraction (see Fact \ref{fact:max}). In the following result we show that $\Invh{X}$ is a normal spectral space and that $F_M$ coincides with the natural retraction to the closed points.

\begin{theorem}\label{thm:normal} The closed points of the spectral topological space $\Invh{X}$ are exactly the finitely satisfiable types. Moreover, the space  $\Invh{X}$ is normal and the natural retraction to the closed points $$r_M^{\rm h}:\Invh{X} \rightarrow \Coh{X}$$ coincides with $F_M$.
\end{theorem}
\begin{proof}Let $p\in \Inv{X}$ be a closed point and let us show that $p$ is finitely satisfiable in $M$. The converse follows from Lemma \ref{cohclosed}. Since 
\[\bigcap_{\theta\in p\text{ honest}}[\theta]^{\rm inv} = \mathrm{cl}^{\rm h}(p)\cap \Inv{X}=\{p\},\] for any $\psi\in p$ there are finitely many honest formulas $\theta_1,\ldots,\theta_{\ell}\in p$ such that 
\[ 
[\theta_1]^{\rm inv}\cap  \cdots \cap [\theta_\ell]^{\rm inv} \subset [\psi]^{\rm inv}.
\] Thus, by Corollary \ref{interhonest} we have  $\emptyset \neq (\theta_1\wedge\cdots\wedge\theta_{\ell})(M)\subset \psi(M)$.
	
To verify that the space is normal, consider two closed non-empty disjoint subsets $C_1$ and $C_2$ of $\Invh{X}$. We can assume that $C_1=[\psi_1]^{\rm inv}$ and $C_2=[\psi_2]^{\rm inv}$ where both $\psi_1$ and $\psi_2$ are a finite conjunction of honest formulas. By Corollary \ref{interhonest}, we have $\psi_1(M)\neq \emptyset$ and $\psi_2(M)\neq\emptyset$. Thus, by Fact \ref{F:Honest_exist} there are two honest $\LL_{\bar{M}}$-formulas $\theta_1(x)$ and $\theta_2(x)$ such that:
\begin{itemize}
	\item $\theta_1(M)=\psi_1(M)=\neg\theta_2(M)$, $\theta_1(\bar{M})\subset \neg\theta_2(\bar{M})$ and
	\item no $M$-invariant type contains the formula $\neg\theta_2\wedge \neg \theta_1$. 
\end{itemize}
In particular, the set $[\theta_1]^{\rm inv}=[\neg \theta_2]^{\rm inv}$ is open and closed in $\Inv{X}$. In addition, we check that $[\psi_i]^{\rm inv}\subset [\theta_i]^{\rm inv}$ for $i=1,2$. Indeed, if $[\psi_1]^{\rm inv}\cap [\theta_2]^{\rm inv}$ were non-empty, then it would follow by Corollary \ref{interhonest} that $\psi_1(M)\cap \theta_2(M)\neq \emptyset$, which is a contradiction. Therefore $[\psi_1]^{\rm inv}\subset [\neg \theta_2]^{\rm inv}=[\theta_1]^{\rm inv}$. Likewise, we get that $[\psi_2]^{\rm inv}\cap [\theta_1]^{\rm inv}=\emptyset$, as otherwise $\psi_2(M)\cap \theta_1(M)$ and hence $\psi_1(M)\cap \psi_2(M)$ would be non-empty. Consequently  $[\psi_2]^{\rm inv}\subset [\neg \theta_1]^{\rm inv}= [ \theta_2]^{\rm inv}$, as required.

Once we have shown that $\Invh{X}$  is a normal spectral space, the retraction $r_M^{\rm h}(p)$ is by definition the unique closed point in $\text{cl}^{\rm h}(p)$, so $r_M^{\rm h}(p)=F_M(p)$ for every $p\in \Invh{X}$. In particular, the map $F_M$ is continuous for the honest topology. 		
\end{proof}	
\begin{remark}
We write down an alternative proof of the normality of $\Invh{X}$ by spectral topological means. By \cite[Prop.\,2]{CC83}, to prove that $ \Invh{X}$ is normal it is enough to show for any $p\in \Invh{X}$ that $\text{cl}^{\rm h}(p)\cap \Coh{X}=\{F_M(p)\}$. Fix $q\in \Coh{X}$ such that $q\in \text{cl}^{\rm h}(p)$ and $q\neq F_M(p)$. Then there is an honest formula $\theta(x)\in q$ such that $\neg \theta(x)\in F_M(p)$. By Proposition \ref{honestF_M} there is an honest formula $\theta_0\in p$ with $\theta_0(M)=\neg \theta(M)$. Since $p\in [\theta_0]^{\rm inv}$, we get $q\in [\theta_0]^{\rm inv}$, so $\theta\wedge \theta_0\in q$, which is a contradiction since $(\theta\wedge \theta_0)(M)=\emptyset$ and $q$ is  finitely satisfiable in $M$. 
\end{remark}

As in the case of the spectral topology, the  closure of a point in $\Invh{X}$ is  totally ordered under specialization. However, the reason why this is true in the honest topology is radically different and concerns the small size of such closure:

\begin{prop} For every $p\in \Invh{X}$ we have that $\mathrm{cl}^{\rm h}(p)=\{p,F_M(p)\}$.	
\end{prop}	
\begin{proof} Assume $p\notin \Coh{X}$, since the points of the latter set are closed. Suppose that there is some $p_1\in \text{cl}^{\rm h}(p)$ with $p_1\neq p$ and $p_1\neq F_M(p)$. Note that 
	\[
	F_M(p) = r_M^{\rm h}(p) = r_M^{\rm h}(p_1) \in \text{cl}^{\rm h}(p_1).
	\] Since the honest topology is $\mathrm{T}_0$ and $p_1\neq F_M(p)$, there is an honest formula $\theta_1(x)$ such that $\theta_1\in F_M(p)$ but $\theta_1 \notin p_1$. 
	We claim: $\text{cl}^h(p)=\text{cl}^h(p_1)$.  
	
	It is enough to show that $\text{cl}^h(p)\subset \text{cl}^h(p_1)$, so pick  an honest formula $\theta(x)$ with $p_1\in [\theta]$.  Thus, $\theta \wedge \theta_1\in F_M(p)$ and so $(\theta\wedge \theta_1)(M)\neq \emptyset$. By Proposition \ref{honestF_M} there is an honest formula $\theta'\in p$ with $(\theta\wedge \theta_1)(M)=\theta'(M)$. Thus $\theta\wedge \theta' \wedge \neg\theta_1\in p_1$, as $p_1\in \text{cl}^{\rm h}(p)$. Note that $(\theta\wedge \theta' \wedge \neg\theta_1)(M)=\emptyset$, so $\theta\wedge \theta' \wedge \neg\theta_1$ is an open set in the honest topology. Hence, as $p_1\in [\theta\wedge \theta' \wedge \neg\theta_1]$, it follows  $p\in [\theta\wedge \theta' \wedge \neg\theta_1]$ and so $p\in [\theta]$, as required.
	
	However, once we have shown that $\text{cl}^h(p)=\text{cl}^h(p_1)$, since the honest topology is $\mathrm{T}_0$, it follows that $p=p_1$, which is a contradiction.
\end{proof}
We finish the paper  pointing out two more differences between the spectral and the honest to\-po\-logy (see Lemma \ref{spectralLocCon} and Lemma \ref{L:DimBeta}).
\begin{lemma} The spectral spaces $\HS_X(\bar\M)$ and $\Invh{X}$ are not locally connected.	
\end{lemma}
\begin{proof} Let $\HS_X(\bar\M)^{\diamond}$ denote one of the two topological space. Suppose that $\HS_X(\bar\M)^{\diamond}$ is locally connected and fix a connected component $C$ of $\HS_X(\bar\M)^{\diamond}$ containing a non-realized $M$-invariant type. As $C$ is open and closed, there are honest $ \LL_{\bar{M}}$-formulas $\theta_1,\ldots,\theta_\ell$ such that $C=[\theta_1]^{\diamond}\cap \cdots \cap [\theta_\ell]^{\diamond}$, where $[\theta_i]^{\diamond}$ denotes $[\theta_i]\cap \HS_X(\bar\M)^{\diamond}$. By Corollary \ref{interhonest} there is some $m\in (\theta_1 \wedge\cdots \wedge \theta_\ell)(M)$. Both the formulas $x=m$ and $\neg (x = m)$ are $M$-definable and therefore they are honest. Since \[C= [\theta_1]^{\diamond}\cap \cdots \cap [\theta_\ell]^{\diamond} \subset [x=m]\cup [\neg (x = m)]\] and $C$ is connected, we deduce that $C$ equals $[x=m]$, which is a contradiction.
\end{proof}	
\begin{lemma} Assume that for every model of the theory the algebraic closure provides a geometry (so that we have a well-defined dimension). Then $\dim(p)=\dim(F_M(p))$ for all $p\in \Invh{X}$. 
\end{lemma}
\begin{proof}For each $p\in \Invh{X}$ we have $p|_M=F_M(p)|_M$. So, clearly it is enough to prove: if $q\in \Invh{X}$ with $\dim(q)=d$, then $\dim(q|_M)=d$.
	
Let $a$ be a realization of $q$ in some elementary extension of $\bar \M$ and let $b=(b_1,\ldots,b_\ell)\in \bar M^{\ell}$ be such that $\dim(b/M)=\ell$ and $\dim(a/M,b)=d$. Suppose $\dim(a/M)>d$. Since 
\[
\dim(a,b/M)=\dim(a/M,b)+\dim(b/M)=\dim(b/M,a)+\dim(a/M),
\] it follows that $\dim(b/M,a)<\dim(b/M)=\ell$. Then there exists an $\LL_M$-formula $\psi(x,y_1,\ldots,y_\ell)$ such that  $\psi(a,b_1,\ldots,b_{\ell-1},y_\ell)$ defines a finite set and $\psi(a,b_1,\ldots,b_\ell)$ holds, so $\psi(x,b_1,\ldots,b_{\ell})\in q$. By $M$-invariance of $q$, for each automorphism $\sigma$ of $\bar\M$ that fixes $M,b_1,\ldots,b_{\ell-1}$ pointwise we have that $\psi(x,b_1,\ldots,b_{\ell-1},\sigma(b_\ell))$ belongs to $q$, so 	$a$ realizes each $\psi(x,b_1,\ldots,b_{\ell-1},\sigma(b_\ell))$. Thus, since $\psi(a,b_1,\ldots,b_{\ell-1},y_\ell)$ only has a finite number of realizations,  we deduce that $b_\ell$ has a finite orbit, so it is in the algebraic closure of $M,b_1,\ldots,b_{\ell-1}$ and hence $\dim(b/M)<\ell$, which is a contradiction. \end{proof}

\end{document}